\documentclass[11pt]{article}
\usepackage{fullpage}
\usepackage[psamsfonts]{eucal}
\usepackage{amsmath}
\usepackage{amsfonts}
\usepackage{dsfont}
\usepackage{caption}
\usepackage{amssymb}
\usepackage{amstext}
\usepackage{amscd}
\newcommand{\bT}{\mathbb{T}}
\usepackage{amsthm}
\usepackage{makeidx}
\usepackage{graphicx}
\usepackage[colorlinks=true,linkcolor=blue]{hyperref}
\usepackage{supertabular}
\usepackage{enumerate}
\newcommand{\cG}{\mathcal{G}}
\usepackage{titling}
\usepackage{wasysym}
\usepackage{color}
\usepackage{xcolor}
\usepackage{comment}
\usepackage{cleveref}
\usepackage{listings}
\usepackage{relsize}

\usepackage{mathptmx}
\usepackage{microtype}
\definecolor{forestgreen}{RGB}{34, 139, 34}

\def\@settitle{\begin{center}
    \bfseries
 \normalfont\LARGE\@title
  \end{center}
}
\def\@setauthors{\begin{center}
 \normalsize\@author
  \end{center}
}
\def\author#1{\par
    {\centering{\authorfont#1}\par\vspace*{0.05in}}
}

\def\titlefont{\fontsize{15}{17}\bfseries\boldmath\selectfont\centering{}}
\def\authorfont{\fontsize{13}{15}}

\let\affiliationfont\rhfont

\def\address#1{\par
    {\centering{\affiliationfont#1\par}}\par\vspace*{11pt}
}

\def\body{
\setcounter{footnote}{0}
\def\thefootnote{\alph{footnote}}
\def\@makefnmark{{$^{\rm \@thefnmark}$}}
}

\def\title#1{
    \thispagestyle{plain}
    \vspace*{-14pt}
    \vskip 79pt
    {\centering{\titlefont #1\par}}
    \vskip 1em
}
\theoremstyle{plain}
\newtheorem{thm}{Theorem}

\newtheorem{prop}[thm]{Proposition}
\usepackage{cite}
\newcommand{\bP}{\mathbb{P}}
\newtheorem{lemma}[thm]{Lemma}

\newtheorem{dfn}[thm]{Definition}

\newtheorem{remark}[thm]{Remark}

\numberwithin{thm}{section}

\newcommand{\fD}{\mathfrak{D}}
\newcommand{\cT}{\mathcal{T}}
\newcommand{\bN}{\mathbb{N}}

\numberwithin{equation}{section}

\newcommand{\bw}{\mathbf{w}}
\newcommand{\bR}{\mathbb{R}}
\newcommand{\Cov}{\textnormal{Cov}}

\usepackage[margin=.75in]{geometry}

\newcommand{\Var}{\textnormal{Var}}

\newcommand{\bC}{\mathbb{C}}

\newcommand{\ri}{\mathrm{i}}
\newcommand{\be}{\begin{equation}}
\newcommand{\ee}{\end{equation}}
\newcommand{\deq}{\mathrel{\mathop:}=} 

\newcommand{\bD}{\mathbb D}

\newcommand{\distr}{\textnormal{distr}}
\newcommand{\bH}{\mathbb H}
\newcommand{\cN}{\mathcal N}
\newcommand{\wt}{\widetilde}
\newcommand{\rd}{{\rm d}}
\newcommand{\oCplus}{\overline{\bC^+}}
\newcommand{\Cplus}{\bC^+}
\newcommand{\re}{{\rm e}}
\newcommand{\bE}{\mathbb{E}}
\newcommand{\spec}{\textnormal{spec}}
\newcommand{\bd}{\mathbf{d}}

\newcommand{\dI}{\mathds{I}}
\begin{document}
\title{The Gaussian Wave for Graphs of Finite Cone Type}

\vspace{1.2cm}

\noindent \begin{minipage}[c]{0.6\textwidth}
 \author{Amir Dembo}
\address{Stanford University\\
   adembo@stanford.edu}
 \end{minipage}
  \begin{minipage}[c]{0.2\textwidth}
 \author{Theo McKenzie}
\address{Stanford University\\
   theom@stanford.edu }
 \end{minipage}
\abstract{We show that for any infinite tree of finite cone type satisfying a mild expansion condition, the only typical process on its vertices with covariance induced by the Green's function is the Gaussian wave. This generalizes a result of Backhausz and Szegedy, who proved this for the infinite regular tree of degree $d\geq 3$. We do this by giving a reduction to a statement concerning the distribution of the inner product of our process with columns of the Green's function, which in turn are straightforward to calculate. 

As a consequence, for random bipartite biregular graphs, the distribution of local neighborhoods of eigenvectors must approximate the Gaussian wave. Moreover, for generic configuration models including random lifts, the local distribution of a uniformly chosen eigenvector from any
arbitrarily small spectral window likewise converges to the Gaussian wave. 
}
\section{Introduction}
Berry's random wave conjecture is a foundational prediction in the study of quantum chaotic systems \cite{berry1977regular,berry1983semiclassical}. It theorizes that the local statistics of eigenfunctions converge to those of a Gaussian random wave. The discrete version of this question is well studied, and large sparse graphs have long been used as a system of quantum chaos \cite{smilansky2013discrete}. Whereas a general heuristic is that these graphs will have their spectral measure converge to their local weak limit, the idea of Berry's conjecture is that, for chaotic models, we can move past this weak convergence to a stronger notion dealing with local statistics.

 We will make this explicit through the \emph{Green's function}. Define $\Cplus\deq\{z\in \bC:\Im[z]>0\}$ and its closure $\oCplus\deq\{z\in \bC:\Im[z]\geq 0\}$. We define the Green's function for a graph $\cG$ with adjacency operator $A_{\cG}$ and $z\in \bC^+$ as
\[
G(z,\cG)\deq [A_{\cG}-z\dI]^{-1}.
\]
We can then extend the definition to the real line by saying that for $\lambda\in \bR$, 
\[
G(\lambda,\cG)\deq \lim_{\eta\rightarrow 0^+}G(\lambda+\ri\eta,\cG).
\]
If $\cG$ is a finite graph on vertex set $[N]$, we define the \emph{Stieltjes transform} $m(z,\cG)$ to be the normalized trace of the Green's function, $\frac1N\sum_{x\in[N]}G_{xx}$. We can think of this as the expected value of the diagonal entry of the Green's function for a vertex selected uniformly at random. More generally, if $\cG$ is a \emph{unimodular random graph} (see \cite{aldous2007processes} for a thorough overview), we can take the expectation over the root of the graph. Therefore, we define the
Stieltjes transform of $\cG$ for $z\in \oCplus$ to be
\[
m(z,\cG)\deq \bE_{o}[G(z,\cG)_{oo}],
\]
where $\bE_o[\cdot]$ refers to the expectation over the randomly selected root according to the unimodular measure.

We say that $\lambda\in \bR$ is in the \emph{purely absolutely continuous spectrum} if 
\be\label{eq:absolutelycontinuous}
0<\Im(m(\lambda))<\infty.
\ee
Whereas point spectrum corresponds to localized states, the absolutely continuous spectrum corresponds to dynamical delocalization of the quantum operator (e.g. \cite[Theorem 2.6]{aizenman2015random}), so it is where we witness chaotic behavior of eigenfunctions from Berry's conjecture. 
The impressive work of Backhausz and Szegedy \cite{backhausz2019almost} proved that every eigenvector, besides the trivial eigenvector of eigenvalue $d$, has the local statistics of a Gaussian process on the infinite regular tree $\cT_d$ (see \Cref{thm:mainbs}). As the spectrum of the adjacency operator on the infinite regular tree is purely absolutely continuous, this gives evidence towards Berry's conjecture.

The $d$-regular graph benefits from a key simplifying feature: its local weak limit is the infinite regular tree, which has a large automorphism group. This is due to it being vertex transitive, and, more strongly, distance regular (for any pair of paths of the same length, there is an automorphism sending one path to the other). Any limiting wave must respect this symmetry, and this invariance plays a decisive role in identifying and solving the associated formulas of a limiting process.

Nevertheless, there is no reason to expect the symmetry of the regular tree to be essential to chaotic behavior, and we prove a result toward Berry’s conjecture for a broader class of random graph models generalizing the random regular graph. Specifically, for a large family of configuration models---random graphs with a prescribed degree distribution---we show that eigenvectors exhibit Gaussian-type statistics, driven solely by local expansion properties. A principal innovation is a structural decomposition of the Gaussian process on the infinite tree via the Green’s function, which makes this mechanism transparent. The Gaussian wave can be realized as a Gaussian process obtained by pushing forward i.i.d. Gaussian noise through the resolvent. Namely, for $z=\lambda+\ri\eta\in\overline{\bC^+}$, graph $\cT$, and a family $\{\xi_x\}_{x\in V(\cT)}$ of i.i.d.\ random variables with $\xi_x\sim\cN(0,1)$, the Gaussian wave $\Psi^z(\cT)$ (see \Cref{dfn:gwave}) admits the representation
\be\label{eq:gaussianprocess}
\Psi^z(\cT)\overset{d}{=} \sqrt{\eta}\sum_{x\in V(\cT)} \xi_x\, G(z,\cT)_x,
\ee
where $G(z,\cT)_x(\cdot)\deq G(z,\cT)_{\cdot x }$ denotes the column of the kernel of $G(z,\cT)$ corresponding to $x\in V(\cT)$.   To show any candidate distribution must be Gaussian, we examine its entropy, and one key insight is that to find the entropy of the Gaussian wave, it is sufficient to consider moments of the form 
\[
\{\langle \Psi^z(\cT),G(z,\cT)_x\rangle\}_{x\in V(\cT)}.
\]
 Of course, because $\Psi^z$ has the decomposition from \eqref{eq:gaussianprocess}, the above is easy to compute.

While the Green's function representation of covariance is known (see \cite[Section 1.1]{he2025gaussian}), treating this decomposition as a primary analytical tool allows for a direct comparison between the cases $z\in \bC^+$ and $z\in\bR$ and provides a transparent framework for analyzing the entropy of the Gaussian wave. \eqref{eq:gaussianprocess} can be compared to the heuristic related to Berry's conjecture that on a hyperbolic manifold, the Fourier coefficients of the Poisson kernel should be i.i.d. Gaussian \cite{aurich1993statistical,zelditch2010recent}. Indeed, we would expect this type of decomposition to hold in many models and to be useful in future work.

In order to precisely state our results, we give the following definition of an eigenvector process with covariance determined by the Green's function. This was defined for the regular tree in \cite{elon2009gaussian}, but a more general definition is as follows. 

\begin{dfn}\label{dfn:gwave}
 An \emph{eigenvector process} $\Phi^{z}(\cT)$ on the (potentially infinite) graph $\cT$ for $z\in \overline{\bC^{+}}$,  is a real-valued process on $V(\cT)$ that is automorphism invariant, is centered, and for every $x,y\in V(\cT)$,  $\bE[\Phi^z(\cT)_x\Phi^z(\cT)_y]=\Im[G(z,\cT)_{xy}]$.

   The \emph{Gaussian wave} $\Psi^{z}(\cT)$ is the unique real valued Gaussian eigenvector process.
\end{dfn}
By the fact that $z\in \overline{\bC^{+}}$, the imaginary part of the Green's function is positive semidefinite, and therefore forms a valid covariance matrix. Moreover, \eqref{eq:gaussianprocess} has the correct covariance by the Ward identity \eqref{eq:wardidentity}. For $z\in \bR$, the Gaussian wave is nontrivial with finite variance if and only if $z$ is in the purely absolutely continuous spectrum. 

For families of finite random graphs, we will be interested in the Gaussian wave on their universal cover, defined as follows.
\begin{dfn}
For any finite graph $\cG$, its \emph{universal cover} $\cT$ is the tree whose vertices are nonbacktracking walks in $\cG$, with adjacency given by one-step extensions of such walks.  
There is a canonical projection $\pi: V(\cT)\to V(\cG)$
that maps each walk to its endpoint in $\cG$.
To obtain a unimodular random rooted tree, we select a root $o$ uniformly at random from $V(\cG)$ and consider the component of $\cT$ corresponding to walks starting at $o$.

A \emph{labeling} of $\cG$ is a function $f:V(\cG)\to Y$, where $Y$ is a separable metrizable topological space.  
Such a labeling induces a labeling $\hat f:V(\cT)\to Y$ on the universal cover by
\[
\hat f(x)\deq f(\pi(x)).
\]
We denote by $\distr^*(f)$ the law of this randomly rooted, labeled universal cover.
\end{dfn}

We study the weak limits of the above defined distributions, therefore, we define the L\'{e}vy-Prokhorov metric for the weak topology (see \cite[Section 11.3]{dudley2018real}).
\begin{dfn}
    For $m\in \bN$ and Borel set $A\in \bR^m$, we denote $A^\epsilon\deq\{x\in \bR^m:\inf_{y\in A}|x-y|<\epsilon\}$.  For two random vectors $X,Y\in \bR^m$, we define the \emph{L\'evy-Prokhorov distance} $d_{LP}(X,Y)$ between $X$ and $Y$ as
    \be\label{eq:levyprokhorov}
    d_{LP}(X,Y)\deq \inf\bigg\{\epsilon\geq 0:\bP( X\in A)\leq \bP(Y\in A^\epsilon)+\epsilon \textnormal{ for all Borel sets }A\bigg\}.
    \ee

    For process $\Phi$ on $\cT$, normalized so that $\Var[\Phi]=1$, we measure its distance to the Gaussian wave $\Psi^\lambda(\cT)$ for $\lambda\in \bR$ and  $k\in \bN$ through the quantity
    \be\label{eq:gaussdiff}
    \Xi_k^\lambda(\Phi)\deq\inf_{0\leq \sigma\leq \Var[\Psi^\lambda]^{-1/2}}\bigg\{d_{LP}(B_k(\Phi^\lambda),B_k(\sigma\Psi^\lambda))\bigg\},
    \ee
    where $B_k(\cdot)$ gives the distribution of the ball of radius $k$ around a randomly selected vertex of $\cT$. 
\end{dfn}
$\Xi_{k}^\lambda$ therefore takes the infimum of the L\'evy-Prokhorov distance of $\Phi$ to any multiple of $\Psi^\lambda$ with variance at most 1. In fact, based on its definition, $\Var[\Psi^z]=\Im[m(z)]$. Utilizing the above definitions, we can state the main theorem of \cite{backhausz2019almost}. 
\begin{thm}[Theorem 2.2 of \cite{backhausz2019almost}]\label{thm:mainbs}
   For any $d\ge 3$, $\epsilon>0$, and $k\in \bN$, there exist $N_0\in \bN$ and $\delta>0$ such that for all $N\ge N_0$, with probability at least $1-\epsilon$ over the choice of a uniformly random $d$-regular graph $\cG$ on $N$ vertices, the following holds: for every vector $\psi$ with $\|\psi\|=1$, if there exists $\lambda\in \bR$ with $\lambda\le d-\epsilon$ such that
$\|(A_{\cG}-\lambda)\psi\|\le \delta$,
then 
$\Xi_k^\lambda\!\left(\distr^*(\sqrt{N}\psi)\right)<\epsilon.$
\end{thm}
By Friedman's theorem \cite{friedman2008proof}, a regular graph selected uniformly at random has, with high probability, all nontrivial eigenvalues with magnitude at most $2\sqrt{d-1}+o_N(1)$. Thus, \Cref{thm:mainbs} encompasses all almost-eigenvectors separated from the all-ones vector. 

In this paper we consider a broader family of graphs.
Specifically, we study random graphs whose universal cover has \emph{finite cone type}, which is a well studied class of models defined as follows.
\begin{dfn}
    For an infinite tree $\cT$ with oriented edge $(o,x)$, define the cone determined by $(o,x)$ to be the connected component of $\cT \backslash {o}$ containing $x$, rooted at $x$.
The tree $\mathcal T$ is said to have \emph{finite cone type} if there are only finitely many isomorphism classes of cones created in this fashion.
\end{dfn}

 To see the utility of this constraint, note that in general, the value of the Green's function on a tree $\cT$ can be deduced from the Schur complement formula (see \eqref{eq:schurcomp}),
\be\label{eq:finitecones}
G(z,\cT)_{oo}=(-z-\sum_{x\sim o} G(z,\cT^{(o)})_{xx})^{-1};\qquad G(z,\cT^{(o)})_{xx}=(-z-\sum_{\substack{y\sim x\\y\neq o}} G(z,\cT^{(x)})_{yy})^{-1}.
\ee
Therefore, given such a tree, it is straightforward to find the behavior of the spectrum at any given $\lambda$. Specifically, for any tree of finite cone type $\cT$, there is a finite discrete set $\fD\in \bR$  such that away from this set, the spectrum and Green's function are well behaved (see \Cref{lem:conetypegf}).

This tractable analysis has led to this class of trees being a fruitful area of study. Initially, this closed form was used to analyze the random walk on these trees \cite{lyons1990random,nagnibeda2002random}.  More recently, the Anderson model on these trees was shown to have purely absolutely continuous spectrum \cite{keller2012absolutely,keller2013spectral} and the eigenfunctions are known to be quantum ergodic \cite{anantharaman2021quantum,anantharaman2019recent,anantharaman2019quantum}. Moreover, in many examples we can \emph{quantify} the portion of the spectrum that is absolutely continuous, both in the specific settings of the the infinite regular tree and the infinite biregular tree, but also more generally \cite{banks2022point,bordenave2017mean,keller2013spectral}.

To create a family of graphs whose universal cover is a given deterministic tree, we use a generalized random model with a fixed set of vertex types. Instead of allowing an arbitrary degree sequence, we fix a finite number of types of vertices and prescribe how many edges each type sends to every other type.

\begin{dfn}
    For $k\in \bN$ fixed, a matrix $\bd = \{\bd_{ij}\}_{i,j=1}^k \in \bN^{k\times k}$ is called an \emph{expanding degree matrix} if:
\begin{enumerate}
    \item when $\bd_{ij} > 0$, then $\bd_{ji} > 0$;
    \item for any $i,j,\ell\in[k]$, if $\bd_{ji}, \bd_{\ell j}, \bd_{\ell i} \neq 0$, then
    \[
      \frac{\bd_{ij}}{\bd_{ji}} \cdot \frac{\bd_{j\ell}}{\bd_{\ell j}}
      = \frac{\bd_{i\ell}}{\bd_{\ell i}};
    \]
    \item $\bd$ is irreducible;
    \item $\max_i \sum_{j} \bd_{ij} \ge 3$ and $\min_i \sum_{j} \bd_{ij} \ge 2$.
    \end{enumerate}

For $\bd$ an expanding degree matrix, we denote by $\cG(N,\bd)$ the random graph ensemble created by partitioning the vertices into sets $V = V_1 \sqcup V_2 \sqcup \cdots \sqcup V_k$, then sampling a graph uniformly at random under the constraint that for every $i,j\in [k]$, the degree of every vertex $x\in V_i$ to $V_j$ is $\bd_{ij}$.

We denote the universal cover of any graph $\cG(N,\bd)$ as $\cT_\bd$, and we denote by $\fD_\bd\subset \bR$ the minimal set such that for all $\lambda \in \spec(A_{\cT_{\bd}})\backslash \fD_\bd$, for every edge $(x,y)\in \cT_{\bd}$, $|G_{xx}(\lambda)|,|G_{xx}^{(y)}(\lambda)|<\infty$, and $\Im[G_{xx}^{(y)}(\lambda)]> 0$, where $\spec(A)$ is the \emph{spectrum} of $A$, that is the set of values $\lambda$ where $(A-\lambda \dI)$ is not invertible. By \Cref{lem:conetypegf}, $\fD_{\bd}$ is finite.
\end{dfn}

 $\cT_{\bd}$ has finite cone type, as the cone is completely determined by entry $\bd_{ij}$ corresponding to the chosen edge $(o,x)$. Also, the size of the sets $|V_i|$ are completely determined by $\bd$, as $|V_i|/|V_j|=\bd_{ji}/\bd_{ij}$ when these entries are nonzero, and by assumption, $\bd$ is irreducible.  In broad terms, our conditions on $\bd$ are as general as possible. The first two ensure that $\bd$ is compatible with a configuration model. The third guarantees that the resulting random graph is connected. The fourth imposes \emph{local expansion}: in the universal cover, every vertex has infinitely many descendants, ensuring that the graph size tends to infinity and spectral chaos can occur. For instance, in the $d$-regular model we require $d\ge 3$ to observe chaotic behavior; similarly, bipartite biregular graphs $\cG(N,d_1,d_2)$ fit into this framework under our constraints as long as $\max\{d_1,d_2\}\geq 3$ and $\min\{d_1,d_2\}\geq 2$.

This framework contains several ubiquitous graph families. Important examples include the following:
\begin{enumerate}[(1)]
    \item \emph{Random $d$-regular graphs} are of the form $\cG(N,d)$.
    \item \emph{Random bipartite biregular graphs} have the form $V = V_1 \sqcup V_2$ and $\bd=\left[\begin{array}{cc}
        0 & d_1 \\
        d_2 & 0 
    \end{array}\right]$. We denote this model by $\cG(N,d_1,d_2)$, where $d_1$ and $d_2$ are the degrees of vertices in $V_1$ and $V_2$, respectively.
    \item 
     \emph{Random lifts of a fixed graph} correspond to the special case where the matrix $\bd$ is the adjacency matrix of an undirected multigraph. The chaotic behavior of the spectrum in lift models is itself an active area of study \cite{bordenave2015new,bordenave2019eigenvalues,chen2024new}.
\end{enumerate}

Although bipartite biregular graphs arise from a configuration model with fixed degrees, they do not arise as a sequence of lifts of a fixed graph when $\bd_{ij} \neq \bd_{ji}$. Their spectral statistics are well studied, and several general results on spectral chaos are known \cite{brito2022spectral,dumitriu2023global,zhu2023second}, motivated in part by applications to coding theory and to the singular vectors of random $0$--$1$ matrices.

For $\cT_{\bd}$ of finite cone type, it is straightforward to characterize the set $\fD_{\bd}$, and in fact $\spec(A_{\cT_\bd})\backslash \fD_\bd$ corresponds to a nontrivial portion of the spectrum for any $\cT_{\bd}$ (see \cite[Section 1.6]{bordenave2017mean} and \cite[Theorem 3.1]{banks2022point}). Intuitively, $\spec(A_{\cT_\bd})\backslash \fD_\bd$ 
is where we would expect chaotic behavior, as it avoids the point spectrum, and also guarantees that spectral statistics on each subtree are nontrivial. For example, by \cite[Section 4.2]{anantharaman2019recent}, $\spec(A_{\cT_\bd})\backslash \fD_\bd$ is exactly the region where we would expect quantum ergodicity, another statistic of eigenvector delocalization.

Our first main result is about the most general model. For this generality we give up control over individual eigenvectors and instead study a randomly selected eigenvector from a small spectral window.

\begin{dfn}\label{dfn:gaussianwindow}
Let $\cG$ be a graph with vertex set $[N]$, and let 
$\{\psi_i\}_{i\in[N]}$ be an orthonormal eigenbasis of $A_\cG$ with
eigenvalues $\{\lambda_i\}_{i\in[N]}$. For $\lambda\in\bR$ and $\epsilon>0$, choose an eigenvalue 
$\lambda_i$ uniformly at random (with multiplicity) from
$\spec(A_\cG)\cap[\lambda-\epsilon,\lambda+\epsilon]$, and let
$\psi_i$ be the associated unit eigenvector. Define the labeling
\[
f^\cG_{\lambda,\epsilon} := \sqrt{N}\,\psi_i .
\]
Also,  define the random almost eigenvector
\[
\wt{\psi}_{\lambda,\epsilon}^{\cG}
\deq
\frac1{\sqrt{|\{\lambda_i\in[\lambda-\epsilon,\lambda+\epsilon]\}|}}
\sum_{\lambda_i\in[\lambda-\epsilon,\lambda+\epsilon]} g_i\,\psi_i,
\]
where the $g_i$ are independent standard Gaussian random variables.

\end{dfn}

Our first result is that the above distributions have Gaussian statistics.

\begin{thm}\label{thm:GWconverge}
For any expanding degree matrix $\bd$, any $k\in\bN$, and any sequence $\epsilon_N=o_N(1)$ that decreases sufficiently slowly, with probability $1-o_N(1)$ over $\cG\sim\cG(N,\bd)$, the following holds: for every $\lambda$ such that $\spec(A_{\cT_{\bd}})\cap[\lambda-\epsilon_N,\lambda+\epsilon_N]$ avoids the finite set $\fD_\bd$, we have $\Xi_k^\lambda\!\left(\distr^*(f^\cG_{\lambda,\epsilon_N})\right)<\epsilon_N$.

The same conclusion holds if $f^\cG_{\lambda,\epsilon_N}$ is replaced by the almost eigenvector $\sqrt{N}\,(\wt{\psi}_{\lambda,\epsilon}^{\cG}+v_N),$ 
where $v_N$ is any sequence satisfying
$\|v_N\|\le\delta_N$ for sufficiently small $\delta_N
>0$.
\end{thm}

   Our result can be viewed as reasoning about the  ``noisy'' distribution of an eigenvector: rather than studying a specific eigenvector $\psi$, we examine the distribution of an eigenvector sampled from an arbitrarily small spectral window. The uniform sampling may be replaced by any distribution that is sufficiently flat over the window (every eigenvector has probability $O(1/N)$ for fixed $\epsilon$).
   
   It is not necessarily true that in this generalized setting every $\lambda\in \spec(\cT_\bd)\backslash \fD_\bd$ has eigenvectors with Gaussian statistics. For example, if we take $\cG(N,\bd)$ to be a random lift, there can be planted eigenvalues with eigenvectors with the same distribution as one in the base graph. For example, every lift of the 4 clique $K_4$ has an (unnormalized) eigenvector of eigenvalue $-1$ where a quarter of entries are $3$, and three quarters of entries are $-1$, despite the fact that $\fD_{\bd}=\emptyset$ for $\bd=d\geq 3$. 

The fact that we must take an average over a window stems from the lack of symmetry in general trees of finite cone type: there is not enough structure to guarantee that the empirical covariance converges to that of the eigenvector process. To have such structure, it is sufficient to impose \emph{radial symmetry}, meaning that if $o$ is a root, then re-rooting the graph at vertices $x$ and $y$ yield isomorphic rooted graphs whenever $\mathrm{dist}(o,x)=\mathrm{dist}(o,y)$. This symmetry condition has enabled other sharp spectral results on infinite trees \cite{keller2013spectral}. In particular, it holds for random biregular graphs, allowing us to prove an analogue of \Cref{thm:mainbs}.

   \begin{thm}\label{thm:GWconverge2}
For any $\epsilon>0$, $k\in\bN$, and degrees $(d_1,d_2)$ with $d_1\ge 3$ and $ d_2\geq 2$, there exist $N_0$ and $\delta>0$ such that for all $N\ge N_0$, with probability at least $1-\epsilon$ over the randomly selected bipartite graph $\cG_N\sim\cG(N,d_1,d_2)$, the following holds: for every $\psi$ with $\|\psi\|=1$ satisfying 
$\|(A_{\cG_N}-\lambda)\psi\|\le \delta$ for $|\lambda|\leq \sqrt{d_1d_2}-\epsilon$  and, if $d_1\neq d_2$, $|\lambda|\geq \epsilon$, we have
\[
\Xi_k^\lambda\!\left(\distr^*(\sqrt{N}\psi)\right) < \epsilon.
\]
\end{thm}
Our condition on $\lambda$ avoids the point spectrum at $0$ and the trivial eigenvalues $\pm \sqrt{d_1d_2}$. In fact, by \cite[Theorem~4]{brito2022spectral},  our condition is equivalent to requiring that for $d_1\geq d_2$, 
$\sqrt{d_1-1}-\sqrt{d_2-1}-o_N(1) \le |\lambda| \le \sqrt{d_1-1}+\sqrt{d_2-1}+o_N(1)$. Moreover, the same idea yields simplifications in the proof of \Cref{thm:mainbs}.

   Our work can be compared to \cite{he2025gaussian}, where the authors show convergence to the Gaussian wave for the random regular graph in the specific case where the eigenvalue is one of the $k$ smallest or $k$ largest for fixed $k$ (therefore making it a statement for eigenvalues satisfying $||\lambda|-2\sqrt{d-1}|=N^{-2/3+o_N(1)}$). In fact, they do not need to take the infimum over all $\sigma$ in \eqref{eq:gaussdiff}, but rather can deterministically set $\sigma=\Var(\Psi^z)^{-1/2}$, thereby showing full $\ell_2$ delocalization. They do this by applying the local law universality of the Green's function for random regular graphs shown in \cite{huang2024ramanujan}. In contrast, our approach does not use any local Green's function law; instead, we show that the Gaussian process in \eqref{eq:gaussianprocess} has an easily analyzable entropy for any tree.

We also remark that \eqref{eq:gaussianprocess} immediately implies that the Gaussian wave is the limit of a linear factor of i.i.d. process for any graph $\cG$ and any $\lambda\in\bR$ lying in the purely absolutely continuous spectrum of $A_{\cG}$ (by adjusting the variance of $\xi_x$, an analogous statement holds when $\lambda$ is in the point spectrum). This extends \cite[Theorem~4]{harangi2015independence}, which established the result under the additional assumption of vertex transitivity as a tool to study the independence ratio on infinite graphs. We therefore expect \eqref{eq:gaussianprocess} to be of independent interest.
\subsection{Key entropy inequality}
The key observation in \cite{backhausz2019almost} is that the empirical distribution of an eigenvector in the configuration model is a \emph{typical process}, meaning that it defines a labeling which, in the weak topology, must be close to a labeling realized by a randomly selected graph, with high probability. 

\begin{dfn}
Let $Y$ be a separable, metrizable topological space, and let $\mu$ be a probability distribution on processes in $Y^{V(\cT_{\bd})}$. We say $\mu$ is \emph{typical} if there exists an increasing sequence $N$ such that for $\cG_N\sim\cG(N,\bd)$, with probability $1$, there are labelings $f_N:V(\cG_N)\to Y$ for which $\distr^*(f_N)$ converges to $\mu$ in the weak topology.
\end{dfn}

Our main theorem is the following.

\begin{thm}\label{thm:mainthm}
For any expanding degree matrix $\bd$ and any $\lambda\in \spec(\cT_{\bd})\backslash \fD_{\bd}$, the only typical eigenvector process on $\cT_\bd$ at $\lambda$ is the Gaussian wave $\Psi^\lambda(\cT_\bd)$.
\end{thm}

The reduction of \Cref{thm:GWconverge} and \Cref{thm:GWconverge2} to \Cref{thm:mainthm} follows from \cite[Proposition~5.1]{backhausz2019almost}; for completeness, we reproduce the proof in \Cref{sec:gen}. The remainder of the paper is devoted to proving \Cref{thm:mainthm}. 

\section{Preliminaries}
\subsection{Notation and conventions}
We list notation that is used throughout the paper. We fix expanding degree matrix $\bd$ and the corresponding universal cover $\cT_\bd$. We denote by $V_\bd$ the vertex set of $\cT_\bd$.  For vertex $u\in V_{\bd}$, we define $\delta_u$ to be the indicator vector on the vertex $u$. For a set of vertices $\Gamma\subset V_\bd$, we define the boundary $\partial \Gamma\subseteq \Gamma$ to be the subset of vertices adjacent to $\cT_{\bd}\backslash \Gamma$.

For $z\in \overline{\bC^{+}}$, in our analysis of the Green's function $G^z\deq G(z,\cT_{\bd})$, we will often consider the function $G_x^z(y)\deq G^z(y,x)$, which is the column of the kernel of $G^z$. More generally, we have $G^z_{\Gamma}\in \bC^{V_{\bd}\times |\Gamma|}$. We will drop the dependence on $z$ when $z$ is clear from the context. We now introduce notation to consider the probability measure on marginals. 

\begin{dfn}
For any separable, metrizable, topological space $Y$ and probability distribution $\mu$ on $Y^{V_{\bd}}$, define $\mu_\Gamma$ to be the marginal distribution on $\Gamma\subset V_\bd$. Under the motivation of our decomposition \eqref{eq:gaussianprocess}, for the probability measure $\mu$ associated with the process $\Phi^\lambda$, we then define $\mu^{\eta}$ to be the probability measure on $\bR^{V_\bd}$, where $u\in V_{\bd}$ is labeled with $-\eta\ri\langle \Phi^\lambda,G^{\lambda+\ri\eta}_u\rangle$.
\end{dfn}

Following \cite{backhausz2019almost}, we reduce the problem to an entropy inequality. Some care is required in defining entropy for the following reasons:

\begin{enumerate}[(1)]
    \item 
As any distribution must satisfy the eigenvector equation at every vertex, the distribution on a connected subgraph $\Gamma$ with $\Gamma\backslash\partial \Gamma\neq \emptyset$ is not full rank.
    \item 
We will want an entropic inequality associated with an arbitrary distribution, whereas Shannon entropy is defined for discrete distributions, and differential entropy for smooth distributions. 
\end{enumerate}

For (1), rather than calculate entropy in the full ambient space, we project onto the intrinsic subspace.  For (2), we discretize the process then take the Shannon entropy of the resulting discrete distribution. The discretization is performed on the full-rank $|\partial \Gamma|$-dimensional subspace rather than coordinatewise, so that entropy reflects the true dimension of the distribution. To do this, we consider a basis that satisfies the eigenvector equation at each vertex in the interior.

\begin{dfn}[Discretization at scale $a$]
For any $\Gamma\subset V_\bd$, we fix an orthonormal basis $\Pi(\Gamma)=(\chi_1,\ldots,\chi_{|\partial\Gamma|})$ of vectors orthogonal to all vectors of the form $\lambda \delta_u-\sum_{v\sim u}\delta_v$ for $u\in \Gamma\backslash \partial \Gamma$. Note the law of $\mu_{\Gamma}$ is determined by
\[
\big(\langle \Phi^\lambda,\chi_1\rangle,\ldots,\langle\Phi^\lambda, \chi_{|\partial\Gamma|}\rangle\big).
\]
Therefore, for $a>0$, we define $\mu_{\Gamma,a}$ to be the discrete probability measure of
\[
\frac1{a}\big(\lfloor a\langle \Phi^\lambda,\chi_1\rangle\rfloor,\ldots,\lfloor a\langle \Phi^\lambda,\chi_{|\partial\Gamma|}\rangle\rfloor\big).
\]
As the quantity we are chiefly interested in is entropy, which is invariant under unitary transformation, it does not matter which basis of vectors we choose.
\end{dfn}

Specifically, the quantity we will be most interested in is the difference of entropies of a star and an edge. Therefore we consider the star $C_o$, which is the union of the vertex $o\in V_\bd$ with its neighbors. We denote the edge between $o$ and $i$ as $e_{oi}$. For a discrete probability measure $\mu$, we consider the Shannon entropy $\bH(\mu)=\bE[-\log(\mu(X))]$, where $X$ is a random variable distributed according to $\mu$. Then our comparison is, for any fixed $k\in \bN$, 
\be\label{eq:entropydef}
\Delta_k(\mu)\deq \lim_{a\rightarrow \infty}\bE_o[\bH(\mu_{B_k(C_o),a})-\frac12\sum_{i\sim o}\bH(\mu_{B_k(e_{oi}),a})],
\ee
where once again, $B_k(\Gamma)$ is the ball of radius $k$ around the set of vertices $\Gamma$. We will often specifically consider $\Delta_0(\mu)= \lim_{a\rightarrow \infty}\bE_o[\bH(\mu_{C_o,a})-\frac12\sum_{i\sim o}\bH(\mu_{e_{oi},a})]$ .

We have the relation (see \cite[Theorem 8.3.1]{cover1999elements}) that for any law $\mu$ with smooth density,
\be\label{eq:deltatoD}
\bH(\mu_{\Gamma,a})=|\partial \Gamma|\log a+\bD(\mu_{\Pi(\Gamma)})+o_a(1).
\ee
where $\bD(\mu)\deq-\int_{\bR^{|\partial \Gamma|}}f_{\mu_{\Pi(\Gamma)}}(x)\log f_{\mu_{\Pi(\Gamma)}}(x)\rd x$ for density function $f_{\mu_{\Pi(\Gamma)}}$ corresponding to $\mu_{\Pi(\Gamma)}$, is the differential entropy on the $|\partial \Gamma|$ dimensional subspace orthogonal to vectors of the form $\lambda\delta_u-\sum_{v\sim u}\delta_v$.

Therefore, when $\mu$ is the law of a smooth eigenvector process, we can equivalently write
\be\label{eq:smoothdelta}
\Delta_k(\mu)=\bE_o[\bD(\mu_{\Pi(B_k(C_o))})]-\frac12\sum_{i\sim o}\bD(\mu_{\Pi(B_k(e_{oi}))})],
\ee
using that the $\log a$ terms cancel out exactly. 
This will be useful specifically when we consider either the Gaussian wave or an eigenvector process smoothed by adding a Gaussian wave component to another eigenvector process. 

\subsection{Proof of \Cref{thm:mainthm}}
In this subsection, we prove \Cref{thm:mainthm} under the assumption of a handful of propositions. In the later sections, we will prove the individual propositions referenced in this section.

To start, our way of quantifying typicality is through the following generalization of  \cite[Theorem 4]{backhausz2018large}.
\begin{prop}\label{lem:typicalentropy}
   Recall the definition of $\Delta_k$ from \eqref{eq:entropydef}. Let $Y$ be a separable, metrizable topological space, and let $\mu$ be a probability distribution on processes in $Y^{V(\cT_{\bd})}$. If $\mu$ is typical, then
    \be\label{eq:typicalentropy}
    \Delta_0(\mu)\geq 0.
    \ee
\end{prop}
Our restrictions on $Y$ are loose enough for \Cref{lem:typicalentropy} to imply $\Delta_k(\mu)\geq 0$ for every $k\in \bN$, as we can label vertex $u$ with the vector of labels in $B_k(u)$, meaning the entropy on a single vertex is that of the entire ball in the original process.  
The next step is to show that for the Gaussian wave, this inequality is an equality. 
 One key observation is that for any $\ell_2$ vector $\chi\in \bR^{V(\cT_\bd)}$, any eigenvector process $\Phi^\lambda$, and any $\eta$,
\be\label{eq:anyeta}
\langle \Phi^\lambda,G^{\lambda+\ri \eta}\chi\rangle=-\ri\eta^{-1}\langle \Phi^\lambda,\chi\rangle.
\ee
This means that for any set of vectors $\chi_1,\ldots\chi_k\in \bR^{V(\cT_{\bd})}$, the joint distribution of the random variables
\be\label{eq:gfvector}
\left\{-\ri\eta\langle \Phi^\lambda,G^{\lambda+\ri \eta}{\chi_i}\rangle\right\}_{i\in [k]}.
\ee
does not depend on $\eta$. Therefore, as we apply a fixed linear transform,  for any eigenvector processes with measures $\mu$ and $\nu$,
\be\label{eq:samedelta}
\Delta_0(\mu^\eta)-\Delta_0(\nu^\eta)=\Delta_0(\mu)-\Delta_0(\nu).
\ee
The advantage of considering different $\eta$ is that the inner product of the vectors $G^{\lambda+\ri \eta}\chi$ changes with $\eta$, adjusting the entropy of \eqref{eq:gfvector}. Specifically, we show the following, which is one of the main novelties of our work. 

\begin{prop}\label{lem:gaussiantyp}
For the Gaussian wave $\Psi^\lambda$ with probability measure $\mu$, recall that the probability measure $\mu^{\eta}$ on $\cT_\bd$ for the distribution where the value at vertex $u$ is $-\ri\eta \langle \Psi^\lambda,G_{u}^{\lambda+\ri \eta}\rangle$. Then for any $k\in \bN$, 
\be\label{eq:gwequality}
  \lim_{\eta\rightarrow 0^+} \Delta_k(\mu^\eta)= 0.
    \ee
\end{prop}

In fact, this proof is general enough that $\Delta$ could be alternatively defined to take the difference in entropy of any two sets of the same boundary size. Next, we show the Gaussian is the unique maximizer of the differential entropy.

\begin{prop}\label{lem:heateq}
$\Delta_0(\mu)$ is uniquely maximized over all eigenvector processes by $\mu$ the probability measure corresponding to $\Psi^\lambda$ the Gaussian wave.
\end{prop}
To do this, we adapt the general argument of \cite[Section 11]{backhausz2019almost}, where we show that heating the random variable increases $\Delta_0$ for any other distribution. However, because there is much less symmetry, there are some required novelties.  In the previous result, the authors use the de Bruijn identity (see \Cref{lem:debruijn}) with an unknown basis, then deduce properties of the basis using symmetries of the graph. In this work, we find a canonical basis based on the Green's function that is well behaved and works for any configuration model.

The final piece is to extrapolate from knowing that the Gaussian is the unique maximizer of $\Delta_0$ to the unique maximizer on $\Delta_k$ for all $k$. For this, we prove the following. This is the identical argument as in \cite[Section 10]{backhausz2019almost} for regular graphs, but we include the proof for completeness in \Cref{sec:levelcompare}.

\begin{lemma}\label{lem:levelcompare}
For any probability measure $\mu$ on $\bR^{V_\bd}$, we have
\[
\Delta_k(\mu)\leq  \Delta_{k-1}(\mu).
\]
We have equality for every $k\in \bN$ if and only if the distribution is \emph{2-Markov}, meaning that conditioned on any edge, the distributions on the two sides of the edge are independent. 
\end{lemma}

\begin{proof}[Proof of \Cref{thm:mainthm}]

Consider a typical eigenvector process $\Phi^{\lambda}(\cT_\bd)$ with probability measure $\mu$. For each $\eta$ and $k$, $\mu_{B_k}^\eta$, the process where we label vertex $u$ with the vector $\{-\ri \eta\langle \Phi^\lambda,G^{\lambda+\ri \eta}_w\rangle\}_{w\in B_k(u)}$, is typical. 
Therefore, by \Cref{lem:typicalentropy}, $\Delta_k(\mu^\eta)\geq 0$ for all $k\geq 0$ and $\eta\geq 0$. Therefore, combining \Cref{lem:gaussiantyp},  \Cref{lem:heateq}, and \eqref{eq:samedelta}, the marginals on stars of $\Phi^\lambda$ are Gaussian. If the distribution is 2-Markov, then the marginals on stars defines the entire distribution, and it is the Gaussian wave. Now, assume our measure $\mu$ is not 2-Markov. We compare it to a measure $\nu$, which is the Gaussian wave.  Then there exists some fixed radius $k$ on which the marginals on the balls of radius $k$ differ, so by \Cref{lem:levelcompare}, $\Delta_k(\mu^\eta)<\Delta_k(\nu^\eta)$ for all $\eta\geq 0$. Using \eqref{eq:anyeta}, \Cref{lem:typicalentropy}, and \Cref{lem:gaussiantyp} again, we recover that $\mu$ is not typical. 
\end{proof}
The remainder of the paper is devoted to proving the propositions introduced in this section. In \Cref{sec:typical}, we explain how typical processes can be characterized through their marginals, as in \Cref{lem:typicalentropy}. This is a straightforward extension of \cite[Theorem 5]{backhausz2018large}, but may give some intuition to those who are not familiar with the topic. In \Cref{sec:gaussiantyp}, we prove \Cref{lem:gaussiantyp} using elementary properties of the Green's function. Finally, in \Cref{sec:heateq} we establish \Cref{lem:heateq} by analyzing the heated random variable through a well-chosen basis and more properties of the Green's function.
\section{Entropic identities for typical processes}\label{sec:typical}

For a discrete measure $\mu$, if there is exists a labeling of the vertices of a graph for which the distribution over vertices is $\mu$, then the number of ways to label a vertex set in such a way that the empirical distribution of vertex labels on $N$ vertices is $\mu$ is $\re^{(1+o_N(1))N\bH(\mu)}$.
Along with labelings of vertices, we now wish to consider labelings on connected pairs $(i,j)$. For edges from $V_i$ to itself, this is done previously. 

\begin{lemma}[Lemma 4.1 in \cite{backhausz2018large}]\label{lem:perfect1}
Let $V$ and $S$ be finite sets and $|V|=N$. Let $\mu_e$ be a probability distribution on $S\times S$ with one dimensional marginal measure $\mu$ in both coordates. Define $M_f$ to be the set of perfect matchings such that the distribution across random edges is $\mu_e$. If $M_f$ is non-empty, it satisfies
\[
|M_f|=(N!!)\re^{(1+o_N(1))N(\frac12\bH(\mu_e)-\bH(\mu))}.
\]
\end{lemma}

We now extend this result to a corresponding bipartite version. 
\begin{lemma}[Bipartite version]\label{lem:perfect2}
Consider two finite vertex sets $|V_1|=|V_2|=N$, and $\mu_e$ a probability distribution on $S\times S$. Define $\mu_1$ and $\mu_2$ to be the marginals on the first and second coordinates with respective distributions $f_1$ and $f_2$. Define $M_{f_1f_2}$ to be the set of perfect matchings from $V_1$ to $V_2$ for which, given a labeled graph, the marginal on edges is $\mu_e$. If $M_{f_1f_2}$ is non-empty, then 
\[
|M_{f_1f_2}|=(N!)\re^{(1+o_N(1))N(\bH(\mu_e)-(\bH(\mu_1)+\bH(\mu_2)))}.
\]
\end{lemma}
\begin{proof}
Consider the set $M'$ of labelings of perfect matchings that satisfy $\mu_e$. We will count $|M'|$ in two ways. In the first way, label with marginals $f_1$ and $f_2$, then match edges. Under this counting, we can see $|M'|=|M_{f_1f_2}|\re^{(1+o_N(1))N(\bH(\mu_1)+\bH(\mu_2))}$. 

The second way to count is that there are $N!$ perfect matchings. Given a perfect matching, counting the number of viable labelings gives $|M'|=(N!)\re^{(1+o_N(1))N\bH(\mu_e)}$. Combining these two equalities gives the result.
\end{proof}

We can now use this to compare a star to edges. 
\begin{proof}[Proof of \Cref{lem:typicalentropy}]
Consider a labeling $f:\vec{E}(\cT_{\bd})\to S\times S$, where $\vec{E}(\cT_{\bd})$ is the set of half-edges of $\cT_{\bd}$. We require that if two half-edges $\vec{e_1}$ and $\vec{e_2}$ originate at the same vertex, then $f(\vec{e_1})$ and $f(\vec{e_2})$ share the same first coordinate. 

We construct the configuration model by matching half-edges so that a half-edge of type $i\to j$ with label $(s_1,s_2)$ is paired with a half-edge of type $j\to i$ labeled $(s_2,s_1)$. Define $\cG_\epsilon(N,\bd,\mu)$ to be the set of labeled graphs such that the star and edge marginals are within $\epsilon$ in total variation of $\mu_{C_i}$ and $\mu_{e_{ij}}$ for all pairs $i,j$. We compare $|\cG_\epsilon(N,\bd,\mu)|$ with the total number of graphs in the configuration model, $|\cG(N,\bd)|$. This is within a constant factor of the total number of simple graphs (see \cite[Theorem 2.2]{wormald1999models}). Define $q_i\deq \frac{|V_i|}{N}$ to be the fraction of vertices of type $i$, we have
\be\label{eq:confignumber}
|\cG(N,\bd)|
=
\left(\prod_{i=1}^k \frac{(\bd_{ii} q_i N)!!}{(\bd_{ii}!)^{q_i N}}\right)
\left(\prod_{1\le i<j\le k} \frac{(\bd_{ij} q_i N)!}{(\bd_{ij}!)^{q_i N}(\bd_{ji}!)^{q_j N}}\right),
\ee
obtained by matching all half-edges while accounting for those of the same type.

We can construct a graph in $\cG_\epsilon(N,\bd,\mu)$ by first randomly sampling the labelings such that they have the correct distribution on stars, for which there are $\re^{(1+o_N(1))\,N\sum_i q_i \bH(\mu_{C_i})}$. We then match edges between $V_i$ and $V_j$ in such a way that respects the matching.  Since edges of type $i\to j$ are independent of edges between other type pairs, we may compute $|\cG_\epsilon(N,\bd,\mu)|$ by evaluating, for each $(i,j)$,  $|M_{ij}|$, the number of matchings within $\epsilon$ of $\mu_{e_{ij}}$. Because the entropy of a half-edge equals that of the union of a half-edge with its paired reverse edge, we apply
\Cref{lem:perfect1} to obtain
\[
|M_{ii}|=\frac{(\bd_{ii} q_iN)!!\,
\re^{(1+o(1))\bH(\mu_{e_{ii}})\,\bd_{ii} q_i N/2}}
{(\bd_{ii}!)^{q_iN}\,
\re^{\bH(\mu_{e_{ii}})\,\bd_{ii} q_i N}},
\]
where $o(1)$ is a term that goes to zero when we send $N$ to infinity, then $\epsilon$ to 0. We have divided by $(\bd_{ii}!)^{q_iN}$ because we do not distinguish edges that stem from the same vertex. Similarly, by \Cref{lem:perfect2},
\[
|M_{ij}|=\frac{(\bd_{ij} q_iN)!\,
\re^{(1+o(1))\bH(\mu_{e_{ij}})\,\bd_{ij} q_iN}}
{(\bd_{ij}!)^{q_iN}(\bd_{ji}!)^{q_jN}\,
\re^{(\bd_{ij}q_i+\bd_{ji}q_j)N\,\bH(\mu_{e_{ij}})}}.
\]
Therefore, to estimate $\cG_\epsilon(N,\bd,\mu)$, we have
\begin{align}
\begin{split}\label{eq:Gepsilonest}
|\cG_\epsilon(N,\bd,\mu)|
&= \exp\left((1+o(1))\,N\sum_i q_i \bH(\mu_{C_i})\right)
   \left(\prod_{1\le i\le k} |M_{ii}|\right)
   \left(\prod_{1\le i<j\le k} |M_{ij}|\right)
\\[0.3em]
&=\exp\!\left((1+o(1))\,N\sum_i q_i\Big(\bH(\mu_{C_i})-\frac12\sum_j \bd_{ij}\bH(\mu_{e_{ij}})\Big)\right)\,
|\cG(N,\bd)|,
\end{split}
\end{align}
where in the last line, we use \eqref{eq:confignumber}.

Since $\mu$ is typical, for any $\epsilon>0$ we have
\be\label{eq:almostyp}
\limsup_{N\to\infty} \frac{|\cG_\epsilon(N,\bd,\mu)|}{|\cG(N,\bd)|} \ge c_{\bd},
\ee
for $c_\bd\geq \exp[(1-(\max_{i}\{\sum_{j}\bd_{ij}\})^2)/4]$ the probability a graph sampled from the configuration model is simple.  Combining \eqref{eq:almostyp} with \eqref{eq:Gepsilonest}  yields \eqref{eq:typicalentropy}.

\end{proof}

\section{Entropy of the Gaussian wave}\label{sec:gaussiantyp}
We first need to confirm the rank of the covariance matrix is indeed the size of the boundary. This will also be a  warm-up for the analysis to come in \Cref{sec:heateq}.
\begin{lemma}\label{lem:fullrank}
    For any $\Gamma\in  \{B_k(C), B_k(e)\}_{k\in \bN}$, the rank of $\lim_{\eta\rightarrow 0^+}\Im(G_{\Gamma}^{\lambda+\ri \eta})$ is $|\partial \Gamma|$.
\end{lemma}
\begin{proof}
    For any $\eta\geq 0$, $\Im[G_{\Gamma}^{\lambda+\ri \eta}]$ is the Gram matrix of $\{\sqrt{\eta}G_x^{\lambda+\ri \eta}\}_{x\in \Gamma}$, by \eqref{eq:wardidentity}. Therefore, it is sufficient to lower bound the minimum singular value of this latter matrix.
We construct an explicit linear transformation giving 
$|\partial\Gamma|$ orthogonal vectors. 
Let $x\in\partial\Gamma$ and let $y\in\Gamma$ denote its parent. 
By the Schur complement formula \eqref{eq:walkdecomp},
\begin{equation}\label{eq:span}
G_x^{(y)}
= G_x - \frac{G_{xy}}{G_{yy}}\, G_y,
\end{equation}
where $G^{(y)}=[A_{\cT^{(y)}}-z\dI]^{-1}$ is the Green's function of the graph with $y$ removed.

If $x_1\neq x_2$ are distinct vertices of $\partial\Gamma$ with respective
parents $y_1,y_2\in\Gamma$, then the vectors 
$G_{x_1}^{(y_1)}$ and $G_{x_2}^{(y_2)}$ are supported on disjoint forward cones,
hence are orthogonal. Therefore, $\{\sqrt{\eta}\, G_{x}^{(y)}\}_{x\in\partial\Gamma}$ 
is an orthogonal family of size $|\partial\Gamma|$, and thus linearly
independent for every $\eta>0$.

Define a linear operator $T$ by prescribing its rows as $
(TG_\Gamma)_x
= G_x - \frac{G_{xy}}{G_{yy}} G_y$ for $x\in\partial\Gamma$, 
so that $T(\sqrt{\eta}\, G_\Gamma)
= \{\sqrt{\eta}\, G_x^{(y)}\}_{x\in\partial\Gamma}$. 
Applying the singular value inequality 
$\sigma_{\min}(M)\ge \sigma_{\min}(TM)/\|T\|$ yields
\[
\sigma_{|\partial\Gamma|}
(\sqrt{\eta}\, G_\Gamma)
\;\ge\;
\frac{
\min_{x\in\partial\Gamma}
\Im[ G_{xx}^{(y)}]
}{
\|T\|
},
\]
where we used that $\|\sqrt{\eta}\, G_x^{(y)}\|^2
=\Im[ G_{xx}^{(y)}]$. 
As $\lambda\notin\mathfrak D_\bd$, each quantity 
$\Im [G_{xx}^{(y)}]$ admits a finite positive limit as $\eta\to 0^+$.
Moreover, $\|T\|$ remains uniformly bounded in this limit since it
depends only on ratios $G_{xy}/G_{yy}$, which also converge. Therefore, the smallest nonzero singular value of $\sqrt{\eta}\, G_\Gamma$
remains bounded away from zero as $\eta\to 0^+$.
\end{proof}

Given \eqref{eq:anyeta} and \eqref{eq:deltatoD}, we can analyze the entropy of the Gaussian by dealing directly  with the differential entropy. 
\begin{proof}[Proof of \Cref{lem:gaussiantyp}]
As a Gaussian vector, the entropy is defined by the covariance matrix $\Sigma$ and the Gram matrix $M$. Specifically, define $\Sigma^\eta$ to be the matrix with $\Sigma_{uv}^\eta=\Cov(\langle  \Psi^\lambda,\eta G^{\lambda+\ri \eta}_u\rangle,\langle  \Psi^\lambda,\eta G^{\lambda+\ri \eta}_v\rangle )$, and define $M$ to be the matrix with $M_{uv}=\eta^2\langle G^{\lambda+\ri \eta}_u, G^{\lambda+\ri \eta}_v\rangle$. Assuming $M$ is full rank, the multivariate Gaussian of dimension $m$ has differential entropy 
\[
\frac{m}2\log(2\pi e)+\frac12\log\det(\Sigma)-\frac12\log\det(M).
\]

Examining $M$, it reduces by a Ward identity \eqref{eq:wardidentity} to
\be\label{eq:Mreduction}
M_{uv}=\eta^2(G^{\lambda +\ri \eta}_u)^*G^{\lambda +\ri \eta}_v=\eta\Im[G^{\lambda+\ri \eta}_{uv}].
\ee
 On the other hand, by \eqref{eq:anyeta}, 
\be\label{eq:sigmadef}
\Sigma_{uv}=\lim_{\zeta\rightarrow 0}\Im(G_{uv}^{\lambda+\ri \zeta}).
\ee
Therefore, for any set $\Gamma$, we can use \eqref{eq:deltatoD} to write
\begin{align*}
&\phantom{{}={}}\bD(\mu_{\Gamma}^\eta)\\
&= \bD(\Psi_{\Pi(\Gamma)})\\
&=\frac{|\partial \Gamma|}{2}\log(2\pi e)+\frac12\log\det(\Sigma_{\Pi(\Gamma)})-\frac12\log\det(M_{\Pi(\Gamma)})\\
&=\frac {|\partial \Gamma|}2\log(2\pi e)+\frac12\log\det(\lim_{\zeta\rightarrow 0^+}\Im(G_{\Pi(\Gamma)}^{\lambda+\ri \zeta}))-\frac{|\partial \Gamma|}{2}\log(\eta)-\log\det(\Im(G_{\Pi(\Gamma)}^{\lambda+\ri \eta})).
\end{align*}

In the difference of entropies $\bE_o \left[\bD(\mu_{B_k(C_o)}^\eta)-\frac12 \sum_{i\sim o} \bD(\mu_{B_k(e_{oi})}^\eta)\right]$,  for $o\sim i$, the boundary vertices of distance $k$ from $i$ and $k+1$ from $o$ appear once on the left-hand side, and twice on the right-hand side, each with prefactor $1/2$. Therefore, the terms $\frac{|\partial \Gamma|}{2}(\log2\pi e-\log \eta)$ cancel exactly, leaving us with
\begin{align}
\label{eq:detcompare}
\begin{split}
  &\phantom{{}={}}\bE_o \left[\bD(\mu_{B_k(C_o)}^\eta)-\frac12 \sum_{i\sim o} \bD(\mu_{B_k(e_{oi})}^\eta) \right]\\
  &=\frac12\bE_o \left[\log\det(\lim_{\zeta\rightarrow 0^+}\Im(G_{\Pi(B_k(C_o))}^{\lambda+\ri \zeta}))-\log\det(\Im(G_{\Pi(B_k(C_o))}^{\lambda+\ri \eta}))\right]\\
  &- \frac{1}{4} \bE_o\left[\sum_{i\sim o}\left(\log\det(\lim_{\zeta\rightarrow 0^+}\Im(G_{\Pi(B_k(e_{ij}))}^{\lambda+\ri \zeta}))-\log\det(\Im(G_{\Pi(B_k(e_{ij}))}^{\lambda+\ri \eta}))\right)\right] .
  \end{split}
\end{align}
By \Cref{lem:fullrank}, $\Pi$ exactly projects away the 0 eigenspace. Thus, the eigenvalues of $\lim_{\zeta\rightarrow 0^+}\Im(G_{\Pi(\Gamma)}^{\lambda+\ri \zeta})$ are uniformly bounded away from 0. Moreover, they are uniformly bounded as $\lambda\notin\fD_{\bd}$.  Therefore, entrywise convergence implies that 
as $\eta\rightarrow 0$, \eqref{eq:detcompare} converges to 0.

\end{proof}
\section{Maximizing the entropy inequality}\label{sec:heateq}
We wish to use a de Bruijn identity for a heated random variable. Therefore, it will be convenient to reduce to only considering smooth processes. 
\begin{prop}\label{prop:smoothmaximizer}
    $\Psi^\lambda$ is the unique maximizer of $\Delta_0(\Phi^\lambda)$ over smooth eigenvector processes $\Phi^\lambda$.
\end{prop}
We first show how this is sufficient. 
\begin{proof}[Proof of \Cref{lem:heateq}]
    Recall that by \eqref{eq:gaussianprocess}, $\Psi^\lambda$ is a weak limit of factor i.i.d. processes. Therefore, if $\Phi^\lambda$ is a typical eigenvector process, by \cite[Lemma C.1]{backhausz2019almost}, so is $\frac1{\sqrt{2}}(\Phi^\lambda+\Psi^\lambda)$, where the sum is of independent processes. This is a smooth eigenvector process, so by \Cref{prop:smoothmaximizer}, it must be Gaussian. As $\frac1{\sqrt{2}}(\Phi^\lambda+\Psi^\lambda)$ and $\Psi^\lambda$ are both Gaussian, $\Phi^\lambda$ must be the Gaussian wave. 
\end{proof}

In the rest of this section, we give the proof of \Cref{prop:smoothmaximizer}. We divide this into steps. First, we want to choose a basis of our probability space that is easy to work with.

\subsubsection*{Decomposition into orthogonal parts}

\begin{figure}\label{fig:fig1}
    \centering
    \includegraphics[width=0.6\linewidth]{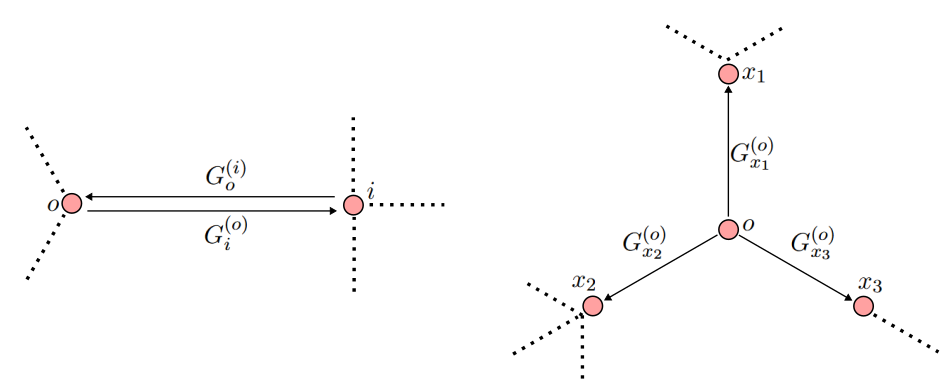}
    \caption{The decomposition of the edge and the star into orthogonal components as given in \eqref{eq:reducedcovar}.}
\end{figure}

 We want a canonical basis to associate with the Green's function. 
Therefore, we define 
\[\xi_y^{(x)}\deq \lim_{\eta\rightarrow 0^+}-\ri \eta\langle \Phi^\lambda,G_y^{(x)}(\lambda+\ri \eta)\rangle=\langle \Phi^\lambda,\delta_y-\frac{G_{yx}}{G_{xx}}\delta_x\rangle.
\]  
Similarly, we can define
\be\label{eq:infinnerprod}\xi_o\deq \lim_{\eta\rightarrow 0}-\ri \eta\langle \Phi^\lambda,G_o\rangle,\qquad \xi_\chi\deq \lim_{\eta\rightarrow 0}-\ri \eta\langle \Phi^\lambda,\chi\rangle.\ee 
where $\chi$ is any linear combination of $\{G_x\}_{x\in V(\cT_\bd)}$. Moreover, we will be interested in the inner product of such $\chi$; therefore, we take our inner product as \[\langle \chi_i,\chi_j\rangle\deq \lim_{\eta\rightarrow 0} \eta \chi_i^*\chi_j,\] which is the correct normalization to guarantee nontrivial inner products. Throughout the rest of this section, unless otherwise specified, we write $G\deq\lim_{\eta\rightarrow 0^+}G^{\lambda+\ri \eta}$. For any $\xi$, we will write the normalized version as
\[
\wt{\xi}\deq \frac{\xi}{\sqrt{\bE[|\xi|^2]}}.
\]

 For $y_1\neq y_2$ that both neighbor $x$, we once again use that $\Im[G_\Gamma^{\lambda+\ri \eta}]$ is the Gram matrix of $\{\sqrt{\eta}G_x^{\lambda+\ri\eta}\}$. This gives that if $y_1,y_2\sim x$ but $y_1\neq y_2$, then
\[
\bE\left[\overline{\xi_{y_1}^{(x)}}\xi_{y_2}^{(x)}\right]=\lim_{\eta\rightarrow 0}\eta\bE[\overline{G_{y_1}^{(x)}(\lambda+\ri \eta)}G_{y_2}^{(x)}(\lambda+\ri \eta)]=0,
\]
as for every $\eta$, $G_{y_1}^{(x)}$ and $G_{y_2}^{(x)}$ are supported on disjoint forward cones. 
Therefore, the $\{\wt{\xi}_x^{(y)}\}$ give us an isotropic basis of both the space spanned by the Green's function of the star $G_C$ and the basis spanned by the Green's function of the edge $G_e$. Since entropy is invariant under fixed full-rank linear changes of coordinates, $\Delta_0$ reduces to a finite-dimensional functional of the joint law of these boundary coordinates. Thus to show \Cref{prop:smoothmaximizer}, it is sufficient to show that the fixed linear transform (full rank by \Cref{lem:fullrank})
\be\label{eq:reducedcovar}
 \bE_o\left[\bD((\wt{\xi}_{i}^{(o)})_{i\sim o})-\frac12 \sum_{i\sim o} \bD(\wt{\xi}_{i}^{(o)},\wt{\xi}_{o}^{(i)})\right]
\ee
is maximized over smooth distributions by the Gaussian.

\subsubsection*{Maximizers must have some independence.}
 We first show that any maximizer has independence between many directions. For any constant $c$,
 \[
\Delta_0(c\Phi^\lambda(G))=\Delta_0(\Phi^\lambda(G)),
 \]
 since differential entropy shifts by $\log|c|$ per coordinate and these shifts cancel in the star–edge difference.
 Therefore, it is sufficient to show that when $\Phi^\lambda$ is not the Gaussian wave,  $\Delta_0(\Phi^\lambda)<\Delta_0(\Phi^\lambda+\sqrt{t}\Psi^\lambda)$ for some $t>0$, where $\Psi^\lambda$ is an independent copy of the Gaussian wave on $\cT_{\bd}$. For this we will use de Bruijn's identity that gives the change in entropy in terms of Fisher information. A simplified version is as follows.
\begin{lemma}\label{lem:debruijn}[Theorem 17.7.2 of \cite{cover1999elements}]
    For any distribution $\hat{\Phi}$ with smooth density function $f:\bR^m\rightarrow \bR$, we add an isotropic Gaussian vector $\hat\Psi$. Then
\[
\frac{\rd}{\rd t} \bD(\hat\Phi+\sqrt{t} \hat\Psi)=\frac12\sum_{i=1}^m\int_{\bR^m}  \frac{\left(\partial_{i} f(x_1,\ldots x_m)\right)^2}{f(x_1,\ldots x_m)} \rd x_1\cdots \rd x_m.
\]
\end{lemma}

To apply this, we need to compare the density function of an edge and the density function of the star. Specifically, we want to compare the change in distribution of $(\xi_{j}^{(o)})_{j\sim o}$ versus that of  $(\xi_{i}^{(o)},\xi_{o}^{(i)})$ as we take $\Phi^\lambda+\sqrt{t}\Psi^\lambda$.
Now fixing $i\sim o$, $\xi_{i}^{(o)}$ is contained in both spaces. On the other hand, to write $\xi_{o}^{(i)}$ in terms of the variables associated with $(\xi_{j}^{(o)})_{j\sim o}$, we note that by \eqref{eq:gffactor}, for $v\in V(\cT_\bd\backslash \{o,i\})$,
\begin{equation}\label{eq:eqforfig}
    G_{ov}^{(i)}=-\sum_{x\sim o, x\neq i} G_{oo}^{(i)}G_{xv}^{(o)}.
\end{equation}
Therefore, as $\lambda\notin \fD_\bd$, $\lim_{\eta\rightarrow 0^+}G_x^{\lambda+\ri \eta}$ has bounded entries, but, by the Ward Identity \eqref{eq:wardidentity}, $\lim_{\eta\rightarrow 0^+}G_x^{\lambda+\ri \eta}$ also has infinite norm. Therefore, 
\be\label{eq:nextstar}
 \xi_{o}^{(i)}=-G_{oo}^{(i)}\sum_{x\sim o, x\neq i} \xi_x^{(o)}.
\ee

 Thus, for our fixed $i\sim o$, we take $U_{oi}$ to be the space $\{\chi\in G_{C_o}:\chi\perp G_{o}^{(i)},G_{i}^{(o)}\}$ (where we mean perpindicular in the sense of \eqref{eq:infinnerprod}). We then will analyze the $d_o$-dimensional density of $\Phi^*G_{C_o}$ by taking $\chi_1,\ldots \chi_{d_o-2}$ as any orthonormal basis of $U_{oi}$, then considering $(\wt{\xi}_{o}^{(i)},\wt{\xi}_{i}^{(o)},\wt{\xi}_{\chi_1},\ldots, \wt{\xi}_{\chi_{d_o-2}})$.
Denote the density function of this random vector as $f(z_1,z_2,\bw)$. We then have
\begin{align}
\begin{split}
&\phantom{{}={}}\int_{\bR^2}  \frac{\left(\partial_{1} f_{e_{oi}}(z_1,z_2)\right)^2}{f_{e_{oi}}(z_1,z_2)} \rd z_1\rd z_2\\
&=\int_{\bR^2}  \frac{\left( \int_{U_{oi}}\partial_{1}f_{C_o}(z_{1},z_2,\bw)\rd \bw\right)^2}{\int_{U_{oi}}f_{C_o}(z_1,z_2,\bw)\rd \bw} \rd z_1\rd z_2\\
&=\int_{\bR^2}  \left( \int_{U_{oi}}\frac{\partial_1f_{C_o}(z_1,z_2,\bw)}{f_{C_o}(z_1,z_2,\bw)}\cdot\frac{f_{C_o}(z_1,z_2,\bw)}{\int_{U_{oi}}f_{C_o}(z_1,z_2,\bw)}\rd \bw\right)^2 \left(\int_{U_{oi}}f_{C_o}(z_1,z_2,\bw)\rd \bw\right)\rd z_1\rd z_2\\
&\leq \int_{\bR^2}  \int_{U_{oi}}\left( \frac{\partial_1f_{C_o}(z_1,z_2,\bw)}{f_{C_o}(z_1,z_2,\bw)}\right)^2\cdot\frac{f_{C_o}(z_1,z_2,\bw)}{\int_{U_{oi}}f_{C_o}(z_1,z_2,\bw)\rd \bw}\rd \bw \left(\int_{U_{oi}}f_{C_o}(z_1,z_2,\bw)\rd \bw\right)\rd z_1 \rd z_2\\
&= \int_{\bR^2}\int_{U_{oi}}  \frac{\left( \partial_1f_{C_o}(z_1,z_2,\bw)\right)^2}{f_{C_o}(z_1,z_2,\bw)}\rd \bw\rd z_1\rd z_2\\
&= \int_{\bR^{d_o}}   \frac{\left( \partial_1f_{C_o}(z_1,z_2,\bw)\right)^2}{f_{C_o}(z_1,z_2,\bw)}\rd \bw\rd z_1\rd z_2.
\end{split}
\end{align}

\begin{figure}\label{fig:fig2}
    \centering
    \includegraphics[width=0.6\linewidth]{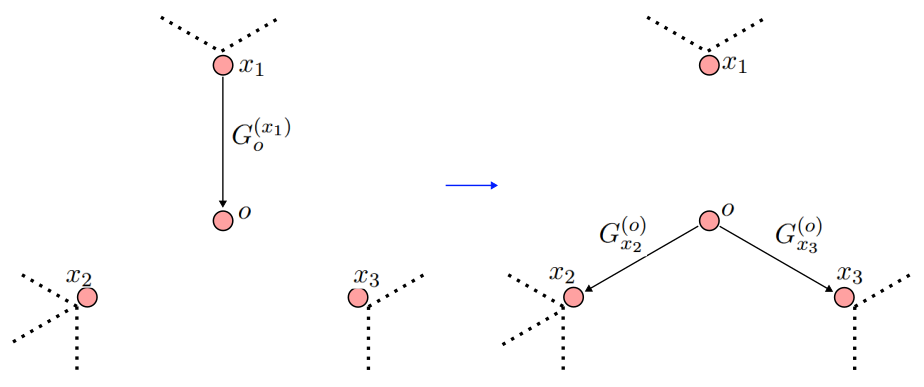}
    \caption{In \eqref{eq:eqforfig}, we can write the Green's function going into a vertex as a sum over Green's functions leaving it.}
    
\end{figure}

In the middle, we use a Jensen inequality. Taking the above inequality, we can take the expectation over half-edges according to the unimodular measure. This then becomes
\begin{align}\label{eq:entropyderivativeineq}
    \bE_o\left[ \sum_{j=1}^{d_o}\int_{\bR^{d_o}}   \frac{\left( \partial_jf_{C_o}(z_1,z_2,\ldots,z_{d_o})\right)^2}{f_{C_o}(z_1,z_2,\ldots,z_{d_o})}\rd \mathbf{z}\right]\geq \frac12\bE_o\left[\sum_{i\sim o} \int_{\bR^{2}}   \frac{\left( \partial_1f_{e_{oi}}(z_1,z_2)\right)^2+\left(\partial_2f_{e_{oi}}(z_1,z_2)\right)^2}{f_{e_{oi}}(z_1,z_2)}\rd z_1\rd z_2\right].
\end{align}

We denote by $\nu_{(t)}$ the probability measure of $\Psi^\lambda+\sqrt{t}\Psi^\lambda$.  For $\Phi^\lambda$ to be typical,  by \Cref{lem:typicalentropy} and \Cref{lem:gaussiantyp},\[\lim_{\zeta\rightarrow 0}\Delta_0(\nu_{(0)}^\zeta)\geq 0= \lim_{t\rightarrow \infty}\lim_{\zeta_{\rightarrow 0}}\Delta_0(\nu_{(t)}^\zeta).\]
By \eqref{eq:entropyderivativeineq}, $t\mapsto \Delta_0(\nu_{(t)})$ is nondecreasing. Therefore, recalling \eqref{eq:samedelta}, for $\Phi^\lambda$ to be typical, \eqref{eq:entropyderivativeineq} must be an equality for every $t$. Thus  for every $t\geq 0$ and every edge $(o,i)$, conditioned on $\xi_{o}^{(i)}$, the variables $\xi_{o}^{(i)}$ and $\{\xi_\chi\}_{\chi\in U_{oi}}$ are mutually independent. 
In fact, we can say something even stronger using the following lemma.

\begin{lemma}[Lemma A.6 of \cite{backhausz2019almost}]\label{lem:edgetypes}
Assume that, for every $0\leq t\leq \infty$, the following holds for the process $\Phi^\lambda+\sqrt{t}\Psi^\lambda$: conditioned on $\xi_{x}^{(y)}$, the random variables $\xi_{y}^{(x)}$ and $\{\xi_\chi\}_{\chi\in U_{xy}}$ are mutually independent. Then one of the following must be true at $t=0$:
\begin{enumerate}
\item $\xi_{x}^{(y)}$ is independent of $\xi_{y}^{(x)}$ and $\{\xi_\chi\}_{\chi\in U_{xy}}$ \quad(Type 1),
\item $\{\xi_\chi\}_{\chi\in U_{xy}}$ is independent of $\xi_{y}^{(x)}$ and $\xi_{x}^{(y)}$ \quad(Type 2).
\end{enumerate}
\end{lemma}

Thus, for every typical process, one of the two alternatives in \Cref{lem:edgetypes} must occur for every edge.

\subsubsection*{Partial independence implies full Gaussianity}

We now want to show that this amount of independence forces the underlying distribution to be Gaussian, for which we use the following result.

\begin{lemma}[Darmois--Skitovich theorem, see Theorem~5.3.1 in \cite{bryc1995normal}]\label{lem:darmois}
Let $\xi_1,\xi_2,\ldots,\xi_d$ be mutually independent random variables. If there exist nonzero coefficients $a_1,b_1,\ldots,a_d,b_d$ such that
$\sum_{i=1}^d a_i\xi_i$ is independent of $\sum_{i=1}^d b_i\xi_i$, then $(\xi_1,\xi_2,\ldots,\xi_d)$ is a Gaussian vector.
\end{lemma}

To apply this, we consider the two options in \Cref{lem:edgetypes}, Type 1 and Type 2. If $(o,i)$ is Type 1, we will usually consider $\xi_{i}^{(o)}$. We now introduce the random variable we usually consider for $(o,i)$ Type 2. In this case, we can do a further decomposition of the Green's function according to \eqref{eq:gffactor} and \eqref{eq:nextstar} to write 
\be\label{eq:gedecomp}
\xi_o= -G_{oo} \xi_{i}^{(o)}+\frac{G_{oo}}{G_{oo}^{(i)}}\xi_{o}^{(i)}.
\ee
Therefore, we take $\chi_i$ to be a unit vector in $\textnormal{span}\{G_o^{(i)},G_i^{(o)}\}$ that is orthogonal to $G_o$. $\chi_i$ is unique up to a unit scalar, and we claim that for every $i,j\sim o$, $\langle\chi_i,\chi_j\rangle\neq 0$. To see this, we can write $G_{i}^{(o)}$ as $c_{1,i}G_o+c_{2,i}\chi_i$, where $c_{1,i}$ and $c_{2,i}$ are nonzero. Then
\[
0=\langle G_{i}^{(o)},G_{j}^{(o)}\rangle=\langle c_{1,i}G_o+c_{2,i}\chi_i,c_{1,j}G_o+c_{2,j}\chi_j\rangle=c_{1,i}c_{1,j}\Im(G_{oo})+c_{2,i}c_{2,j}\langle \chi_i,\chi_j\rangle.
\]
Since $c_{1,i}c_{1,j}\Im(G_{oo})>0$, the second term must be nonzero to cancel it, and hence $\langle\chi_i,\chi_j\rangle\neq0$.

Given this, assume that $o$ has degree at least 3. We then split our analysis into different scenarios.
\begin{itemize}
    \item There are at least two directed edges, say $(o,i)$ and $(o,j)$ of Type 2. Then $\chi_{ij}\deq \chi_i-\langle\chi_i,\chi_j\rangle \chi_{j}\in U_{oj}$, meaning that $\xi_{\chi_j}$ and $\xi_{\chi_{ij}}$ are mutually independent. Similarly, so are $\xi_{\chi_i}$ and $\xi_{\chi_{ji}}$. Therefore, by \Cref{lem:darmois}, as $d_o\geq 3$, $\chi_{ij}$ and $\chi_{ji}$ are nontrivial, the random variables $\xi_\chi$ are Gaussian for $\chi$ in the span of $\chi_w$ for $(o,w)$ Type 2. Moreover, as $\xi_{\chi_{ij}}$ and $\xi_{\chi_{ji}} $ are independent of $\xi_o$, $\xi_o$ is independent of $\xi_\chi$ for $\chi$ in this space.
    \item 
    There is an edge $(o,i)$ of Type 1 and an edge $(o,j)$ of Type 2. We know $\xi_{i}^{(o)}$ is independent of $ \xi_{o}^{(j)} -\langle G_{i}^{(o)},G_{o}^{(j)}\rangle\xi_i^{(o)}$.  
Similarly, $\xi_{o}^{(j)}$ is independent of $\xi_{i}^{(o)}-\langle G_{i}^{(o)},G_{o}^{(j)}\rangle\xi_o^{(j)}$ as $(o,j)$ is of Type 2. Therefore, by \Cref{lem:darmois}, the space spanned by $\xi_i^{(o)}$ of Type 1 and $\xi_o^{(j)}$ of Type 2 is Gaussian. 
\end{itemize}

Combining these scenarios, if $d_o\geq 3$, and there is at least 1 edge of Type 1 and 2 edges of Type 2, the space spanned by $G_{i}^{(o)}$ for $(o,i)$ Type 1 and $\chi_j$ for $(o,j)$ Type 2 is Gaussian. We claim that this is in fact the entire space. To see this, we will create a basis by sequentially adding vectors. We start with $\{\chi_j\}_j$ for $(o,j)$ Type 2. This is full rank, as adding the vector $G_o$ would span a linear space which is also spanned by $\{G_j^{(o)}\}_j$ and $G_o$, which is full rank as long as there is at least one edge $(o,i)$ of Type 1. Next, the space spanned by $\{\chi_{j}\}_j$ cannot contain any $G_i^{(o)}$, as it only contains vectors orthogonal to $G_o$. Therefore, when we sequentially add vectors of the form $G_{i}^{(o)}$ for $(o,i)$ Type 1 to the basis, each of these is linearly independent from all vectors in our partial basis, as $G_{i}^{(o)}$ and $G_{\ell}^{(o)}$ are orthogonal when $i\neq \ell$.

If instead there is exactly one edge $G_{i}^{(o)}$ of Type 2, since every other edge is Type 1, the independence relations force the remaining Type 2 edge to satisfy the Type 1 condition as well.

We are then left to show Gaussianity under very specific conditions on the vertex $o$, which we call Type A and Type B:
\begin{enumerate}[(A)]
    \item All directed edges $(o,i)$ in $C_o$ are of Type 1, in which case we need to show all variables $\xi_{i}^{(o)}$ are Gaussian (we call this Type A); or
    \item All edges are  Type 2, in which case it is sufficient to show that $\xi_o$ is Gaussian (Type B). 
\end{enumerate}

We start by examining $\xi_{x}^{(o)}$ for a star $C_o$ of Type A. Then we use \eqref{eq:nextstar} to decompose it into variables associated with the star centered at $x$. If the degree of $x$ is at least 3, this is always a nontrivial (meaning multiple nonzero coefficients) decomposition into independent components, either into $\xi_{y}^{(x)}$ (for $C_x$ Type A) or into $\xi_{x}$ and a Gaussian (for $C_x$ Type B). If $x$ is degree 2, by \eqref{eq:nextstar} we can write $\xi_{x}^{(o)}$ as a scalar times $\xi_{y}^{(x)}$, for $y$ the unique neighbor of $x$ not equal to $o$. We can continue to perform this transition until we reach another vertex of degree at least 3, yielding the same decomposition. 

If $C_o$ is Type B, then we can expand \eqref{eq:gedecomp}  with reference to a specific neighbor $x$
to write
\begin{align}
    \begin{split}
\label{eq:nextstar2}
\xi_{o}&= -G_{oo} \xi_{{x}}^{(o)}+\left(\frac{G_{oo}}{G_{oo}^{({x})}}\right)\xi_{o}^{({x})}\\
&=G_{oo}G_{xx}^{(o)}\sum_{w\neq {o}, w\sim {x}}\xi_{w}^{(x)}+\left(\frac{G_{oo}}{G_{oo}^{(x)}}\right)\xi_{o}^{(x)}.
\end{split}
\end{align}
Assume the degree of $x$ is at least 3. If $C_x$ is Type A, then \eqref{eq:nextstar2} immediately gives a nontrivial decomposition. 
If $C_x$ is Type B,
we compare \eqref{eq:nextstar2} to
\[
\xi_x=-G_{xx}\sum_{w\neq o, w\sim x}\xi_{w}^{(x)}-G_{xx}\xi_{o}^{(x)}.
\]
Therefore, we can write $\xi_o$ in terms of $\xi_x$ and an independent Gaussian. 
Similarly, if the degree of $x$ is 2, we can do the above decomposition again along the path $o,x,\ldots,x_r=x$, to the next vertex of degree at least 3, as 
\be\label{eq:fardecomp}
\xi_o=-G_{oo}\left(\prod_{i=1}^{r}\left(-G_{x_i}^{(x_{i-1})}\right)\right)\sum_{w\neq {o}, w\sim {x}}\xi_{w}^{(x)}-\left(\frac{G_{oo}}{\prod_{i=1}^{r}\left(-G_{x_{i-1}}^{(x_i)}\right)}\right)\xi_{o}^{(x)}.
\ee
In this case, we have this is a nontrivial decomposition unless 
\begin{align*}
\prod_{i=1}^{r} G_{x_i}^{(x_{i-1})}=\left(\prod_{i=1}^{r}G_{x_{i-1}}^{(x_i)}\right)^{-1},
\end{align*}
which would force the two coefficients in \eqref{eq:fardecomp} to cancel and give a trivial decomposition. However, taking our path to go from one copy of $o$ to another (which must exist as $\bd$ is irreducible), if all vertices on this path of degree at least 3 are Type B, then by 
\Cref{lem:smallcoefficients}, we must have \[\left|\prod_{i=1}^{r} G_{x_i}^{(x_{i-1})}\right|<1<\left|\left(\prod_{i=1}^{r}G_{x_{i-1}}^{(x_i)}\right)^{-1}\right|,\] and there is some decomposition on this path into nontrivial parts.

Therefore, for vertices of degree at least 3, we take $\wt{\xi}_{i}^{(o)}$ for vertices of Type A and $\wt{\xi}_{o}$ on vertices of Type B. As $\cT_\bd$ has finite cone type, these form a finite system of linear equations in distribution of variance 1 random variables. Writing $\zeta_i$ as one of these variables, we can write this as
\[
\zeta_i\overset{d}{=}\sum_{j} \alpha_{i,j} \zeta_j+\alpha_{i} \zeta
\]
where each $|\alpha_{i,j}|<1$, $|\alpha_i|<1$ and $\zeta$ is an independent Gaussian. We claim that every variable here must be Gaussian. To see this, take $\zeta_i$ to be the variable with the smallest differential entropy. Then, by the Entropy Power Inequality (see \cite[Theorem 17.8.1]{cover1999elements}), writing it as the right hand side necessarily has higher differential entropy unless $\zeta_i$ is Gaussian. 

If the degree of $o$ is 2, then we can use \eqref{eq:nextstar} to reduce $\xi_{i}^{(o)}$ to a scalar times the nearest vertex of degree of at least $3$, meaning these variables are also Gaussian. This finishes the proof of \Cref{prop:smoothmaximizer}.

\appendix
\section{Moving from the finite to infinite graph}\label{sec:gen}
\subsection{Convergence of mean and covariance}\label{sec:covariance}
We show that the mean and covariance of the relevant random variables defined in \Cref{thm:GWconverge} and \Cref{thm:GWconverge2} converge to the desired limits. For the general configuration model, we show that $f_\lambda,\epsilon^\cG$ has the correct statistics. An identical proof works for $\wt{\psi}_{\lambda,\epsilon}^{\cG}$ as well.
\begin{lemma}\label{lem:zeromean}
    Let $\bd$ be an expanding degree matrix and let $\cG\sim \cG(N,\bd)$. 
    Recall the distribution $f_{\lambda,\epsilon}^{\cG}$ from \Cref{dfn:gaussianwindow}. 
    Then for any fixed $\epsilon>0$, and any part $V_i$ of the configuration model, if $o$ is sampled uniformly at random from $V_i$, we have $
        \bE_o\!\left[f_{\lambda,\epsilon}^{\cG}(o)\right]=O\!\left(1/N\right).$
\end{lemma}

\begin{proof}
The matrix $\bd$ is the quotient adjacency matrix corresponding to the equitable partition $
    V(\cG)=V_1\sqcup\cdots\sqcup V_k$, 
where $k$ is fixed independently of $N$.  
For the set of block-constant vectors $
    F\deq\bigl\{f:V(\cG)\to\bR \,:\, f \text{ is constant on each } V_i\bigr\}$,
 $\dim F = k$, and $F$ is invariant under the adjacency operator $A_\cG$.    
If $\bd$ is $k\times k$, then $\cG$ has precisely $k$ eigenvectors lying in $F$, obtained by lifting eigenvectors of $\bd$ to block-constant functions on $V(\cG)$.  
All remaining eigenvectors are orthogonal to $F$.

Since $\lambda$ lies in the absolutely continuous spectrum, the window defining $f_{\lambda,\epsilon}^{\cG}$ contains $\Theta(N)$ eigenvalues.  
Thus $f_{\lambda,\epsilon}^{\cG}$ is obtained by averaging over linearly many eigenvectors.  
As only $O(1)$ of these can lie in $F$, the contribution of the block-constant subspace to this average is $O(1/N)$.   Taking the expectation over $o$ sampled uniformly from $V_i$ therefore yields
 $   \bE_o\!\left[f_{\lambda,\epsilon}^{\cG}(o)\right]=O\!\left(1/N\right)$,
as claimed.
\end{proof}

We can now deal with the covariance. 
\begin{lemma}\label{lem:weakconverge}
    We fix a root $o$ for $\cT_\bd$, and let $i$ be any vertex at finite distance from $o$. Then, for $\cG\sim\cG(N,\bd)$, we consider the distribution over the randomly selected root $o_N$, which is a vertex selected uniformly at random across all vertices which $o$ covers. We then set $i_N$ to be the corresponding vertex which $i$ covers.
    
For the function $f^\cG_{\lambda,\epsilon}$ from \Cref{dfn:gaussianwindow}, we have 
    \[
\lim_{N\rightarrow \infty}\frac{\bE_{o_N,i_N}[f^\cG_{\lambda,\epsilon}(o_N)f^\cG_{\lambda,\epsilon}(i_N)]}{\bE_{o_N}[f^\cG_{\lambda,\epsilon}(o_N)^2]}= \frac{\Im[G_{oi}]}{\Im[G_{oo}]},
    \]
    
    \end{lemma}

\begin{proof}
    
    For the configuration model, we use the well known fact, (see, e.g. \cite[Theorem 4.1]{van2024random}), that a configuration model of fixed sparsity converges locally to its local weak limit. Therefore, for every fixed integer $k$, \be\label{eq:treeconverge}\lim_{N\rightarrow \infty }\bE_o[\langle \delta_{i_N},A_{\cG}^k\delta_{o_N}\rangle]=\bE_o[\langle \delta_{i},A_{\cT_\bd}^k\delta_{o}\rangle].
    \ee Therefore, define
    \[\mu_{oi}^N\deq\frac1{N}\sum_{\lambda\in \spec(A_{\cG})}\bE_{o}[\psi^N_{\lambda}(o_N)\psi^N_{\lambda}(i_N)]\delta_\lambda,\]
    and $\mu_{oi}$ defined by its moments
    \[
    \int_{\bR}x^k \rd \mu_{oi}=\langle \delta_{i},A_{\cT_\bd}^k\delta_{o}\rangle.
    \]
    Since the spectrum is uniformly bounded and moments determine the measure, convergence of all moments implies weak convergence of $\mu_{oi}^N$ to $\mu_{oi}$.

By the fact that $\lambda$ is in the absolutely continuous part of the spectrum, by the Stieltjes inversion formula (see \cite[Equation 2.11]{aizenman2015random}), for sufficiently small $\epsilon$,
\[
\int_{\lambda-\epsilon}^{\lambda+\epsilon} \rd\mu_{oo}(x)\geq \frac12\Im[G_{oo}]\epsilon.
\]
Moreover, by weak convergence, for any $\alpha>0$ there is some sufficiently large $N$ such that
\be\label{eq:treeconvergence2}
\left|\int_{\lambda-\epsilon}^{\lambda+\epsilon} \rd\mu_{oi}^N(x)-\rd\mu_{oi}(x)\right|,\left|\int_{\lambda-\epsilon}^{\lambda+\epsilon} \rd\mu_{oo}^N(x)-\rd\mu_{oo}(x)\right|< \alpha.
\ee
Therefore, as $\bE_o[f^\cG_{\lambda,\epsilon}(o_N)f^\cG_{\lambda,\epsilon}(i_N)]=\int_{\lambda-\epsilon}^{\lambda+\epsilon}\rd\mu_{oi}^N(x)$,
by \eqref{eq:treeconvergence2}, for $\alpha\leq \frac14\Im[G_{oo}]\epsilon$, we use that $\frac{x+\alpha}{y-\alpha}-\frac{x}{y}=\frac{\alpha(x+y)}{y(y-\alpha)}$, giving
\be\label{eq:ratioapprox}
\left|\frac{\bE_o[f^\cG_{\lambda,\epsilon}(o_N)f^\cG_{\lambda,\epsilon}(i_N)]}{\bE_o[f^\cG_{\lambda,\epsilon}(o_N)^2]}-\frac{\int_{\lambda-\epsilon}^{\lambda +\epsilon} \rd\mu_{oi}(x)}{\int_{\lambda-\epsilon}^{\lambda +\epsilon} \rd\mu_{oo}(x)}\right|\leq 16\Im[G_{oo}]^{-1}\epsilon^{-1} \alpha.
\ee
Therefore, as $\alpha$ is arbitrary, fixing $\epsilon$ and sending $\alpha$ to 0, we see this converges to 0. 
Thus, \eqref{eq:absolutelycontinuous} and \eqref{eq:ratioapprox} give 
\[
\lim_{\epsilon\rightarrow 0}\lim_{N\rightarrow \infty}\frac{\int_{\lambda-\epsilon}^{\lambda +\epsilon} \rd\mu_{oi}^N(x)}{\int_{\lambda-\epsilon}^{\lambda +\epsilon} \rd\mu_{oo}^N(x)}=\frac{\Im[G_{oi}]}{\Im[G_{oo}]}.
\]

\end{proof}

 We now give a similar decomposition for the bipartite biregular graph. 
 \begin{lemma}\label{lem:zeromeanreg}
     For $d_1\geq d_2$, $d_1\geq 3$, and $d_2\geq 2$, the following is true about the model $\cG\sim \cG(N,d_1,d_2)$.
    \begin{enumerate}[(1)]
    \item 
    The spectrum of the bipartite biregular infinite tree $\cT_{d_1,d_2}$ for $d_1>d_2$ is purely absolutely continuous away from $\fD_{d_1,d_2}=0$. If $d_1=d_2$, it is purely absolutely continuous with $\fD_{d_1,d_2}=\emptyset$.
        \item All eigenvectors of eigenvalue $|\lambda|<\sqrt{d_1d_2}$ are orthogonal to the set of vectors that are constant on each of the two parts $V_1\sqcup V_2$ of the configuration model. 
    \end{enumerate}
 \end{lemma}
 \begin{proof}
 (1) follows immediately from the decomposition of the spectrum in \cite[Corollary 4.5]{godsil1988walk}. For (2) we can do a  similar argument to  \Cref{lem:zeromean}, as the $\pm\sqrt{d_1d_2}$ eigenspaces correspond to vectors which are constant on each part.
 \end{proof}
 \begin{lemma}\label{lem:weakconvergereg}
Fix $k\geq 0$. For the infinite biregular tree $\cT_{d_1,d_2}$, we fix a root $o$ for $\cT_{d_1,d_2}$, and let $i$ be any vertex at distance $k$ from $o$. Then, for the random bipartite biregular graph $\cG\sim\cG(N,d_1,d_2)$, we consider the distribution over the randomly selected vertex $o_N$ covered by $o$, and take $i_N$ to be the image of $i$ in the projection onto $\cG$. For any $\lambda\in \spec(\cT_{\bd})$ with any eigenvector $\psi$, \[
   \frac{\bE_{o_N,i_N}[\psi_{o_N}\psi_{i_N}]}{\bE_{o_N}[\psi_{o_N}^2]}= \frac{\Im[G_{oi}]}{\Im[G_{oo}]}.
    \]
\end{lemma}
\begin{proof}
We evaluate this covariance by lifting $\psi$ to $\distr^*(\psi)$. Because the infinite biregular tree is radially symmetric, it suffices to compute the covariance between $o\in V$ and any vertex at distance $k$, since all such covariances agree by automorphism invariance. Let $S_o(k)$ denote the set of vertices at distance exactly $k$ from $o$. Then
\[
\bE[\psi_o\psi_i]
=\frac{1}{|S_o(k)|}\,
\bE\!\left[\psi_o\sum_{w\in S_o(k)}\psi_w\right].
\]

Assume vertices at even distance from $o$ have degree $d_1$ and those at odd distance have degree $d_2$. Repeatedly applying the eigenvector equation yields the recurrence
\be\label{eq:biregularrec}
\left(\begin{array}{c}
    \displaystyle\sum_{y\in S_o(2j+1)} \psi_y \\
    \displaystyle\sum_{y\in S_o(2j)} \psi_y
\end{array}\right)
=
\left[
\left(\begin{array}{cc}
0 & 1 \\[0.2em]
\lambda & -(d_1-1)
\end{array}\right)
\left(\begin{array}{cc}
0 & 1 \\[0.2em]
\lambda & -(d_2-1)
\end{array}\right)
\right]^j
\left(\begin{array}{c}
\lambda\,\psi_o \\
\psi_o
\end{array}\right).
\ee

This means that $\bE[\psi_k\psi_o]/\bE[\psi_o^2]$ can be expressed as an explicit polynomial in $\lambda,d_1,d_2$, and that this must be satisfied by any eigenvector. As $\lim_{\eta\to0}\|(A-(\lambda+i\eta)\dI)(G_o)\|/\|G_o\|=0$, the same recurrence holds for the Green's function as well.

\end{proof}
\begin{remark}
    This proof also works for the regular tree (or indeed any graph with radial symmetry) by replacing \eqref{eq:biregularrec} with the correct degree sequence. 
\end{remark}

\subsection{Convergence to the Gaussian wave}
\label{sec:prevresults}
 \begin{proof}[Proof of \Cref{thm:GWconverge}]
 By \Cref{lem:zeromean}, the process is centered, and by \Cref{lem:weakconverge}, the covariance matrix of our process is $\sigma^2\Im[G]$, for some $0\leq \sigma^2\leq \frac1{\Im[m(\lambda)]}$. Therefore, it is sufficient to show that any convergent subsequence is typical. 
    
    Define $\mathcal{P}$ to be the set of Borel probability measures on $V(\cT_\bd)$ with second moment at most 1. As $\mathcal{P}$ is tight, it is compact under the metric $d_{LP}$ defined in \eqref{eq:levyprokhorov}. We then wish to define a corresponding metric on graphs. For any graph $\cG$, we can take the set $S(\cG)\deq\{\distr^*(f):f\in\bR^{V(\cG)},\frac1N\sum_{i=1}^Nf(i)^2\leq 1\}$, which once again, is compact in the weak topology. Therefore, we can define $d_N(\cG_1,\cG_2)=d_H(S(\cG_1),S(\cG_2))$, where $d_H$ is the Hausdorff distance under $d_{LP}$. Equivalently, $d_N(\cG_1,\cG_2)$ is the infimum over $\delta$ such that for any labeling $f_1$ of $\cG_1$ satisfying $\frac1N\sum_{i\in [N]}f_1(i)^2\leq 1$, there is a labeling $\hat{f}_1$ of $\cG_2$ such that $d_{LP}(\distr^*(f_1),\distr^*(\hat{f}_1))\leq \delta$, and for every $f_2$ a labeling of $\cG_2$, there is a labeling $\hat{f}_2$ of $\cG_1$ such that $d_{LP}(\distr^*(f_2),\distr^*(\hat{f}_2))\leq \delta$. 
    
   The Hausdorff distance of a compact metric between closed subsets is compact. Therefore
there exists a finite $\epsilon/3$ cover $M$ of the sets $S(\cG)$. This means there is set of subsets $T$ such that with probability at least $1/|M|$, our randomly selected graph $\cG$ has $d_H(S(\cG),T)\leq \epsilon/3$. Therefore, define the random variable $X(\cG)$ to be  $d_H(S(\cG),T)$ for the randomly selected graph $\cG\sim \cG(N,\bd)$. When we are performing a switch, we change the neighborhood of any fixed vertex with probability $O(1/N)$ as $N\rightarrow \infty$. Therefore, for L\'evy-Prokhorov distance, the change caused by this shift is at most $O(1/N)$. This means we can use a McDiarmid inequality on the Doob martingale (see \cite[Theorem 2.19]{wormald1999models} for an explicit derivation) to show that for some fixed $c>0$, with probability at least $1-\exp(-cN^{1/3})$, $|X(\cG)-\bE[X(\cG)]|\leq N^{-1/3}$. Therefore, for sufficiently large $N$ we must have $\bE[X]\leq \epsilon/2$. meaning, by the triangle inequality, with probability $1-\exp(-cN^{1/3})$, $d_N(\cG_1,\cG_2)\leq \epsilon$. By Borel-Cantelli II, this means that if there is an infinite subsequence that converges to a measure, then in fact this measure is $\epsilon$-typical. As $\epsilon$ is arbitrary, this implies that the limit is typical. 

 \end{proof}
The proof of \Cref{thm:GWconverge2} is identical to that of \Cref{thm:GWconverge}, we just replace \Cref{lem:zeromean} and \Cref{lem:weakconverge} with \Cref{lem:zeromeanreg} and \Cref{lem:weakconvergereg}.

\subsection{Comparing levels}\label{sec:levelcompare}
\begin{proof}[Proof of \Cref{lem:levelcompare}]
We use the fact that $\bD(X,Z)+\bD(Y,Z)-\bD(Z)\geq \bD(X,Y,Z)$, with equality if and only if $X$ and $Y$ are conditionally independent given $Z$. This gives the following equations, where $d_o$ is the degree of $o$:
\begin{align}\begin{split}\label{eq:levelwiseentropy}
\bD(\mu_{B_k(C_o)})&\leq \left(\sum_{i\sim o}\bD( \mu_{B_{k}(e_{oi})})\right)-(d_i-1)\bD(\mu_{B_{k-1}(C_o)})\\
\bD(\mu_{B_{k}(e_{oi})})&\leq \bD( \mu_{B_{k-1}(C_o)})+\bD( \mu_{B_{k-1}(C_i)})-\bD(\mu_{B_{k-1}(e_{oi})}).
\end{split}
\end{align}
The first equation in \eqref{eq:levelwiseentropy} follows as we can cover $B_{k}(C_o)$ with $d_o$ balls, one for each edge emerging from the central vertex. This will create an excess overlap of $d_o-1$ balls of radius $k-1$. We have equality if and only if the $d_o$ disconnected components of $B_k(C_o)\backslash B_{k-1}(C_o)$ are conditionally independent. 

Similarly, for the second equation of \eqref{eq:levelwiseentropy}, we can cover $B_{k}(e_{oi})$ with two balls surrounding each of the vertices, with an overlap of a ball of radius $k-1$. We only have equality if, when conditioned on $B_{k-1}(e_{oi})$, the two sides of $B_{k}(e_{oi})\backslash B_{k-1}(e_{oi})$ must be independent.

Therefore, taking an average over the root, we obtain
\begin{align}\begin{split}\label{eq:entropylevelshift}
&\phantom{{}={}}\bE_o\left[ \bD(\mu_{B_k(C_o)})-\frac12 \sum_{i\sim o}  \bD (\mu_{B_{k}(e_{oi})}) \right]\\
&\leq \bE_o\left[\sum_{i\sim o}\bD( \mu_{B_{k}(e_{oi})})-(d_o-1)\bD(\mu_{B_{k-1}(C_o)})-\frac12 \sum_{i\sim o}\bD( \mu_{B_{k}(e_{oi})})\right]\\
&=\bE_o\left[\frac12\left(\sum_{i\sim o}\bD( \mu_{B_{k}(e_{oi})})\right)-(d_o-1)\bD(\mu_{B_{k-1}(C_o)})\right]\\
&\leq\bE_o\left[\frac12\left(\sum_{i\sim o}\bD(\mu_{B_{k-1}(C_o)})+\bD(\mu_{B_{k-1}(C_i)})-\bD(\mu_{B_{k-1}(e_{oi})})\right)-(d_o-1)\bD(\mu_{B_{k-1}(C_o)})\right].
\end{split}
\end{align}
We now track the number of terms $\bD(\mu_{B_{k-1}(C_o)})$ in \eqref{eq:entropylevelshift}. We have $\frac12\sum_i{\bd_{oi}}-d_o+1$ from type $i$, and $\frac12\sum_{i}\bd_{oi}$ from neighbors of $o$. As $\sum_{i}\bd_{oi}=d_o$, the last row of \eqref{eq:entropylevelshift} reduces to 
\[
\bE_o\left[ \bD(\mu_{B_{k-1}(C_o)})-\frac12 \sum_{i\sim o}  \bD (\mu_{B_{k-1}(e_{oi})}) \right].
\]
We have equality for every $k$ if and only if we have equality in \eqref{eq:levelwiseentropy} for every $k$, which means that the distribution is 2-Markov.
\end{proof}
\section{Properties of the Green's Function}
Through this section, let $G=(H-z\dI)^{-1}$  denote the Green's function of Hermitian operator $H$ at $z=E+\ri \eta$. We recall the fundamental recursive properties of the Green's function. The first is the Schur complement formula. 
\begin{lemma}[Lemma 7.2 of \cite{erdHos2017dynamical}]\label{lem:schurcomp}
    For any subset of indices $\bT$, 
    \be\label{eq:schurcomp}
G|_{\bT}=(H|_\bT-z \dI-H|_{\bT \bT^c} G^{(\bT)}H|_{\bT^c \bT})^{-1}.
    \ee
\end{lemma}

The next is the decomposition of the Green's function.

\begin{lemma}[Lemma 8.3 of \cite{erdHos2017dynamical}] The following is true for the Green's function.
\begin{enumerate}
    \item 
    For $y,w\neq x$
    \be\label{eq:walkdecomp}
    G_{yw}^{(x)}=G_{yw}-\frac{G_{yx}G_{xw}}{G_{xx}}.
    \ee
    \item 
    If $x\neq y$, then
    \be\label{eq:gffactor}
    G_{xy}=-G_{xx}\sum_{w\neq x}H_{xw}G_{wy}^{(x)}.
    \ee
    \item (Ward Identity)
    We have that
    \be\label{eq:wardidentity}
    \sum_y |G_{xy}|^2=\frac{\Im(G_{xx})}{\eta}.
    \ee
    \end{enumerate}
\end{lemma}

We collect the following results from \cite[Proposition 4.2]{anantharaman2019recent}.
\begin{lemma}\label{lem:conetypegf}
\begin{enumerate}[\rm (1)]
\item For any tree $\cT$ of finite cone type, there is a finite set $ \fD_1\subset \bR$ such that for every edge $(x,y)\in \cT$, $G_{xx}^{(y)}(\lambda+\ri \eta),G_{xx}(\lambda+\ri\eta)$ is finite for all $\lambda\in \bR\setminus  \fD_1$. The map $\lambda\mapsto G_{xx}^{(y)}(\lambda+i0)$ is continuous on $\bR\setminus \fD_1$.
\item Moreover,
there is a finite set $\fD_2$ such that for $\lambda\in \spec(\cT)\backslash (\fD_1\cup \fD_2)$, for every edge $(x,y)$, $\Im[G_{xx}^{(y)}]>0$ and $\Im[G_{xx}]>0$. 
\end{enumerate}
\end{lemma}

We also a general property of the Green's function for generically chosen $z$.
\begin{lemma}\label{lem:smallcoefficients}
    For expanding degree matrix $\bd$, $\lambda\in\spec(\cT_\bd)\backslash \fD_\bd$, consider $G\deq \lim_{\eta\rightarrow 0}G^{\lambda+\ri\eta}$. Assume there is an automorphism on $\cT_{\bd}$ mapping directed edge $(x,y)$ to directed edge $(x',y')$ and the path has at least one vertex of degree 3. Then for the path $x=\gamma_0,\gamma_1,\ldots,\gamma_t=x'$, 
    \[
    \prod_{i\in [t]} |G_{\gamma_i\gamma_i}^{(\gamma_{i-1})}|< 1.
    \]
\end{lemma}
\begin{proof}
By taking the norm of the Green's function, we have that by iterating \eqref{eq:gffactor},
    \[
    \|G_y^{(x)}\|^2\geq \left(\prod_{i\in [t]} |G_{\gamma_i\gamma_i}^{(\gamma_{i-1})}|^2\right)\|G_{y}^{(x)}\|^2+\|X\|^2
    \]
    where $X$ has nonzero norm since at least one new forward cone is introduced. This immediate implies the result.
\end{proof}

 \subsection*{Acknowledgements.}
The research of A.D. is supported in part by NSF Grant DMS-1954337 and Grant DMS-2348142.

\bibliographystyle{plain}
\bibliography{ref}

\end{document}